\documentclass[11pt,a4paper]{article}
\usepackage[cp1251]{inputenc}
\usepackage{amsmath,amsthm}
\usepackage{amsfonts,amstext,amssymb,verbatim,epsfig}
\usepackage{dsfont}
\usepackage{enumerate}
\usepackage{psfrag}
\usepackage{graphicx}
\usepackage{color}
\usepackage{float}
\usepackage{subfig}

\usepackage[top=3cm, bottom=3.5cm, left=2.5cm, right=2.5cm]{geometry}

\def\R{{\mathbb{R}}}

\def\Z{{\mathbb{Z}}}

\def \p{{\bf P1}} 
\def \pp{{\bf P2}} 
\def \ppp{{\bf P3}} 
\def \s{{\bf S1}} 
\def \sss{{\bf S2}}
\def \a{{\bf A1}}
\def \aa{{\bf A2}}
\def \aaa{{\bf A3}}
\def \aaaa{{\bf A4}}
\def \aaaaa{{\bf A5}}

\def \GG{{\mathbb G}} 
\def \ballZ{{\mathrm B}} 
\def \atom{\omega} 
\def \set{{\mathcal S}} 

\def \ballS{{\ballZ_{\scriptscriptstyle \set}}}

\def \constP{{\chi_{\scriptscriptstyle {\mathrm P}}}}
\def \constS{{\Delta_{\scriptscriptstyle {\mathrm S}}}}
\def \funcP{f_{\scriptscriptstyle {\mathrm P}}}
\def \funcS{f_{\scriptscriptstyle {\mathrm S}}}
\def \RP{R_{\scriptscriptstyle {\mathrm P}}}
\def \RS{R_{\scriptscriptstyle {\mathrm S}}}
\def \LP{L_{\scriptscriptstyle {\mathrm P}}}
\def \epsP{{\varepsilon_{\scriptscriptstyle {\mathrm P}}}}

\def \scexp{{\theta_{\scriptscriptstyle {\mathrm sc}}}}
\def \ei{{\theta_{\scriptscriptstyle {\mathrm{iso}}}}}

\def \setg{{\mathbf M}_A}
\def \setd{{D_A}}
\def \setgs{{\mathbb A}_s}

\newtheorem{theorem}{Theorem}[section]
\newtheorem{lemma}[theorem]{Lemma}
\newtheorem{proposition}[theorem]{Proposition}

\newtheoremstyle{likedef}
  {}%
  {}%
  {}%
  {}
  {\bfseries}%
  {.}%
  {.5em}%
  {}%

\theoremstyle{likedef}

\newtheorem{definition}[theorem]{Definition}
\newtheorem{remark}[theorem]{Remark}

\numberwithin{equation}{section}

\begin{document}

\title{Quenched invariance principle for simple random walk on clusters in correlated percolation models}

\author{
    $\begin{array}{cc}\text{Eviatar B. Procaccia}\\ \text{UCLA}\end{array}$
    $\begin{array}{cc}\text{Ron Rosenthal}\\\text{ETH Z\"urich}\end{array}$
    $\begin{array}{cc}\text{Art\"em Sapozhnikov}\\\text{MPI Leipzig}\end{array}$}

\date{}

\maketitle

\begin{abstract}
We prove a quenched invariance principle for simple random walk on the unique infinite percolation cluster for a general class 
of percolation models on $\Z^d$, $d\geq 2$, with long-range correlations introduced in \cite{DRS12}, 
solving one of the open problems from there. 
This gives new results for random interlacements in dimension $d\geq 3$ at every level, 
as well as for the vacant set of random interlacements and the level sets of the Gaussian free field in the regime of the so-called local uniqueness 
(which is believed to coincide with the whole supercritical regime). 

An essential ingredient of our proof is a new isoperimetric inequality for correlated percolation models. 
\end{abstract}

\section{Introduction and results}

Quenched invariance principles and heat kernel bounds for random walks 
on infinite percolation clusters and among i.i.d. random conductances in $\Z^d$ 
were proved during the last two decades 
(see \cite{Lawler83,MFGW85,MFGW89,SS04,BergerBiskup,MathieuPiatnitski,BiskupPrescott,Mathieu08,BD10,BD11,GZ12,ABDH13, ADS13}). 
The proofs of these results strongly rely on the i.i.d structure of the models and 
some stochastic domination with respect to super-critical Bernoulli percolation.

Many important models in probability theory and in statistical mechanics, 
in particular, models which come from real world phenomena, 
exhibit long range correlations and offer an incentive to create new tools capable of handling models with dependent structures. 
In recent years interest arose in understanding such systems, 
both in specific models such as random interlacements, 
vacant set of random interlacements and the Gaussian free field, 
as well as in general systems, see for example \cite{BLM, DRS12, PR, PoTe, RodSz, SidoraviciusSznitman_RI, SznitmanAM, Sznitman:Decoupling}. 
In the context of invariance principle with long range correlations 
one should emphasize the results of Biskup \cite{Biskup} and Andres, Deuschel, and Slowik \cite{ADS13}, 
that prove a quenched invariance principle for random walk in ergodic random conductances 
under some moment assumptions and ellipticity. 
In this paper we prove a quenched invariance principle for random walks on percolation clusters 
(i.e., in the non-elliptic situation) in the axiomatic framework of Drewitz, R\'ath, and Sapozhnikov \cite{DRS12}. 
This framework encompasses percolation models with strong correlations, including 
random interlacements, vacant set of random interlacements, 
and level sets of the Gaussian free field. 

The main novelty of our proof is a new isoperimetric inequality for correlated percolation models, see Theorem~\ref{thm:isoperimetric}.
We should emphasize that existing methods for proving isoperimetric inequalities (see, e.g., \cite{Barlow,BM03,BBHK,MathieuRemy,Pete08}) only apply to models which 
allow for comparison with Bernoulli percolation after a certain coarsening procedure. 
A common feature of the three examples above is that they cannot be effectively compared with Bernoulli percolation on {\it any} scale. 
Thus, the existing methods for proving isoperimetric inequalities do not apply. 
Our approach is more combinatorial in nature. It does not rely on any ``set counting'' arguments 
and the Liggett-Schonmann-Stacey theorem \cite {LSS}, 
and can be applied to models which do not dominate supercritical Bernoulli percolation after any coarsening.

\subsection{The model}

We consider a one parameter family of probability measures $\mathbb P^u$, $u\in(a,b)\subseteq \R_+$, on the measurable space $(\{0,1\}^{\Z^d},\mathcal F)$, $d\geq 2$, 
where the sigma-algebra $\mathcal F$ is generated by the canonical coordinate maps $\{\atom\mapsto\atom(x)\}_{x\in\Z^d}$.
The numbers $0\leq a<b$ as well as the dimension $d\geq 2$ are going to be fixed throughout the paper, 
and we omit the dependence of various constants on $a$, $b$, and $d$. 

For $x=(x(1),\dots,x(d))\in \R^d$, the $\ell^1$ and $\ell^\infty$ norms of $x$ are defined in the usual way by 
$|x|_1 = \sum_{i=1}^d|x(i)|$ and $|x|_\infty = \max\{|x(1)|,\ldots|x(d)|\}$, respectively.

For any $\atom\in\{0,1\}^{\Z^d}$, we define 
\[
\set = \set(\atom) = \{x\in\Z^d~:~\atom(x) = 1\} \subseteq \Z^d .\
\]
We view $\set$ as a subgraph of $\Z^d$ in which the edges are drawn between any two vertices of $\set$ within $\ell^1$-distance $1$ from each other. 
For $r\in [0,\infty]$, we denote by $\set_r$, the set of vertices of $\set$ 
which are in connected components of $\set$ of $\ell^1$-diameter $\geq r$. 
In particular, $\set_\infty$ is the subset of vertices of $\set$ which are in infinite connected components of $\set$.

An event $G \in \mathcal F$ is called \emph{increasing} (respectively, \emph{decreasing}), if
for all $\atom\in G$ and $\atom' \in \{0,1\}^{\Z^d}$ with $\atom(y) \leq \atom'(y)$ 
(respectively, $\atom(y) \geq \atom'(y)$) for all $y\in\Z^d$, one has $\atom' \in G$.

For $x \in \Z^d$ and $r \in \R_+$, we denote by 
$\ballZ(x,r) = \{y\in\Z^d~:~|x-y|_\infty\leq \lfloor r \rfloor \}$ the closed $l^{\infty}$-ball in $\Z^d$ with 
radius $\lfloor r \rfloor$ and center $x$.

\bigskip

We assume that {\it the measures $\mathbb P^u$, $u\in(a,b)$, satisfy the axioms from \cite{DRS12}}, which we now briefly list. 
The reader is referred to the original paper \cite{DRS12} for a discussion about this setup.

\begin{itemize}
\item[\p{}] {\it (Ergodicity)}
For each $u\in(a,b)$, every lattice shift is measure preserving and ergodic on $(\{0,1\}^{\Z^d},\mathcal F,\mathbb P^u)$.
\item[\pp{}] {\it (Monotonicity)}
For any $u,u'\in(a,b)$ with $u<u'$, and any increasing event $G\in\mathcal F$, 
$\mathbb P^u[G] \leq \mathbb P^{u'}[G]$.
\item[\ppp{}] {\it (Decoupling)}
Let $L\geq 1$ be an integer and $x_1,x_2\in\Z^d$. 
For $i\in\{1,2\}$, let $A_i\in\sigma(\{\atom\mapsto\atom(y)\}_{y\in \ballZ(x_i,10L)})$ be decreasing events, and 
$B_i\in\sigma(\{\atom\mapsto\atom(y)\}_{y\in \ballZ(x_i,10L)})$ increasing events. 
There exist $\RP,\LP <\infty$ and $\epsP,\constP>0$ such that for any integer $R\geq \RP$ and $a<\widehat u<u<b$ satisfying 
\[
u\geq \left(1 + R^{-\constP}\right)\cdot \widehat u ,\
\]
if $|x_1 - x_2|_\infty \geq R\cdot L$, then 
\[
\mathbb P^u\left[A_1\cap A_2\right] \leq 
\mathbb P^{\widehat u}\left[A_1\right] \cdot
\mathbb P^{\widehat u}\left[A_2\right] 
+ e^{-\funcP(L)} ,\
\]
and
\[
\mathbb P^{\widehat u}\left[B_1\cap B_2\right] \leq 
\mathbb P^u\left[B_1\right] \cdot
\mathbb P^u\left[B_2\right] 
+ e^{-\funcP(L)} ,\
\]
where $\funcP$ is a real valued function satisfying $\funcP(L) \geq e^{(\log L)^\epsP}$ for all $L\geq \LP$.
\item[\s{}] {\it (Local uniqueness)}
There exists a function $\funcS:(a,b)\times\Z_+\to \mathbb R$ such that for each $u\in(a,b)$,
\begin{equation}\label{eq:funcS}
\begin{array}{c}
\text{there exist $\constS = \constS(u)>0$ and $\RS = \RS(u)<\infty$}\\
\text{such that $\funcS(u,R) \geq (\log R)^{1+\constS}$ for all $R\geq \RS$,}
\end{array}
\end{equation}
and for all $u\in(a,b)$ and $R\geq 1$, the following inequalities are satisfied:
\begin{equation*}
\mathbb P^u\left[ \, 
\set_R\cap\ballZ(0,R) \neq \emptyset \,
\right]
\geq 
1 - e^{-\funcS(u,R)} , 
\end{equation*}
and
\begin{equation*}
\mathbb P^u\left[
\begin{array}{c}
\text{for all $x,y\in\set_{\scriptscriptstyle{R/10}}\cap\ballZ(0,R)$,}\\
\text{$x$ is connected to $y$ in $\set\cap\ballZ(0,2R)$}
\end{array}
\right]
\geq 1 - e^{-\funcS(u,R)} . 
\end{equation*}
\item[\sss{}] {\it (Continuity)}
Let $\eta(u) = \mathbb P^u\left[0\in\set_\infty\right]$. The function $\eta(\cdot)$ is positive and continuous on $(a,b)$. 
\end{itemize}
Note that if the family $\mathbb P^u$, $u\in(a,b)$, satisfies \s{}, then a union bound argument gives that for any $u\in(a,b)$, 
$\mathbb P^u$-a.s., the set $\set_\infty$ is non-empty and connected, 
and there exist $c_{\scriptscriptstyle 1}=c_{\scriptscriptstyle 1}(u)>0$ and $C_1 = C_1(u)<\infty$ such that for all $R\geq 1$, 
\begin{equation}\label{eq:C1:infty}
\mathbb P^u\left[ \,
\set_\infty\cap\ballZ(0,R) \neq \emptyset \, 
\right]
\geq 
1 - C_1 e^{-c_{\scriptscriptstyle 1}(\log R)^{1+\constS}} . 
\end{equation}
We will comment on the use of conditions \pp{}, \ppp{}, and \sss{} in Remark~\ref{rem:conditions}.

\subsection{Results}

For $\atom\in\{0,1\}^{\Z^d}$ and $x\in\set$, let $\deg_\atom (x) =\left|\{y\in\set ~:~ |y-x|_1=1 \}\right|$ be the degree of $x$ in $\set$, and let 
$\mathbf P_{\atom,x}$ be the distribution of the random walk $\{X_n\}_{n \geq 0}$ on $\set$ defined by the transition kernel
\[
\mathbf P_{\atom,x}[X_{n+1} = z|X_{n}=y] =
\left\{
    \begin{array}{ll}
        \frac{1}{2d}&\quad |z-y|_1=1,~z\in\set;\\
        1-\frac{\deg_\atom(y)}{2d}&\quad z=y;\\
        0 & \quad \text{otherwise,}
    \end{array}
    \right.
\]
and initial position $\mathbf P_{\atom,x}[X_0=x]=1$. 
For $n\in \mathbb N$, and $t\geq 0$, define 
\[
\widetilde B_n(t) = \frac{1}{\sqrt n}\left(X_{\lfloor tn\rfloor} + (tn - \lfloor tn\rfloor)\cdot (X_{\lfloor tn\rfloor + 1} - X_{\lfloor tn\rfloor})\right) .\
\]
Denote by $C[0,T]$ the space of continuous functions from $[0,T]$ to $\mathbb R^d$ equipped with supremum norm, 
and by $\mathcal W_T$ the Borel sigma-algebra on $C[0,T]$.
Our main result is the following theorem.
\begin{theorem}\label{thm:ip}
Let $d\geq 2$, and assume that the family of measures $\mathbb P^u$, $u\in (a,b)$, satisfies assumptions \p{} -- \ppp{} and \s{} -- \sss{}. 
Then for all $u\in(a,b)$, $T>0$, and for $\mathbb P^u[\cdot~|~0\in\set_\infty]$-almost every $\atom$, 
the law of $(\widetilde B_n(t))_{0\leq t\leq T}$ on $(C[0,T],\mathcal W_T)$ converges weakly 
to the law of a Brownian motion with zero drift and non-degenerate covariance matrix. 
In addition, if reflections and rotations of $\Z^d$ by $\frac{\pi}{2}$ preserve $\mathbb P^u$, 
then the limiting Brownian motion is isotropic (with positive diffusion constant).
\end{theorem}
The proof of Theorem~\ref{thm:ip} is based on the well-known construction of the corrector. 
Moreover, it closely follows the proofs of the main results in \cite{BergerBiskup,BiskupPrescott} using \cite[Theorem~1.3]{DRS12}
about chemical distance in $\set$, and Theorem~\ref{thm:isoperimetric} below, which is the main novelty of this paper. 
\begin{theorem}\label{thm:isoperimetric}
Let $d\geq 2$ and $\ei>0$. 
For $A\subset \set$, let $\partial_\set A$ be the edge boundary of $A$ in $\set$, i.e., 
the set of edges from $\Z^d$ with one end-vertex in $A$ and the other in $\set\setminus A$. 
For $R\geq 1$, let $\mathcal C_R$ be a largest connected component (in volume, with ties broken arbitrarily) of $\mathcal S\cap\ballZ(0,R)$. 

If the family of measures $\mathbb P^u$, $u\in (a,b)$, satisfies assumptions \p{} -- \ppp{} and \s{} -- \sss{}, 
then for each $u\in(a,b)$, there exist $\gamma_{\scriptscriptstyle \ref{thm:isoperimetric}} = \gamma_{\scriptscriptstyle \ref{thm:isoperimetric}}(u)>0$, 
$c=c(u,\ei)>0$, and $C=C(u,\ei)<\infty$ such that for all $R\geq 1$, 
\begin{equation}\label{eq:isoperimetric}
\mathbb P^u\left[
\begin{array}{c}
\text{for any $A\subset\mathcal C_R$ with $|A|\geq R^{\ei}$,}\\ 
\text{$|\partial_\set A| \geq \gamma_{\scriptscriptstyle \ref{thm:isoperimetric}}\cdot |A|^{\frac{d-1}{d}}$}
\end{array}
\right]\geq 1 - Ce^{-c(\log R)^{1+\constS}} .\ 
\end{equation}
\end{theorem}
\begin{remark}
\begin{itemize}\itemsep0pt
\item[(1)]
As we will see in the proof of Theorem~\ref{thm:isoperimetric}, under assumptions \p{} -- \ppp{} and \s{} -- \sss{}, 
for each $u\in(a,b)$, with $\mathbb P^u$-probability $\geq 1 - Ce^{-c(\log R)^{1+\constS}}$, there is 
a unique cluster of largest volume in $\set\cap\ballZ(0,R)$.
\item[(2)]
Note that we consider here the boundary of $A$ in $\set$, and not in $\mathcal C_R$. 
This is enough for our purposes. 
The first proofs of the quenched invariance principle for simple random walk on the infinite cluster of Bernoulli percolation \cite{BergerBiskup,MathieuPiatnitski,SS04} 
crucially relied on the Gaussian upper bound on $\mathbf P_{\atom,0}[X_n = x]$ obtained in \cite{Barlow}. 
To prove the desired bound (as well as the corresponding Gaussian lower bound) one needs to show that with 
$\mathbb P^u$-probability $\geq 1 - Ce^{-c(\log R)^{1+\constS}}$, 
the boundary {\it in $\mathcal C_R$} of any $A\subset\mathcal C_R$ such that $|A|\leq \frac 12\cdot |\mathcal C_R|$ has size $\geq c\cdot R^{-1}\cdot |A|$, 
see, e.g., \cite[Proposition~2.11]{Barlow}. 
Thanks to simplifications obtained in \cite{BiskupPrescott}, we do not need to prove such a statement in order to deduce Theorem~\ref{thm:ip}. 
Showing that the Gaussian bounds on the transition density hold under assumptions \p{} -- \ppp{} and \s{} -- \sss{} remains an open problem. 
\item[(3)]
In fact, we do not need the full strength of Theorem~\ref{thm:isoperimetric} to prove Theorem~\ref{thm:ip}, see 
assumption \aaaaa{} in Section~\ref{sec:ipa}.
\item[(4)]
Theorem~\ref{thm:isoperimetric} implies that under the assumptions \p{} -- \ppp{} and \s{} -- \sss{}, 
for any $u\in(a,b)$ and $\mathbb P^u[\cdot~|~0\in\set_\infty]$-almost every $\atom$, 
there exists $K_u = K_u(\atom)<\infty$ such that 
for all $n\geq 1$ and $x\in\Z^d$, $\mathbf P_{\atom,0}[X_n = x] \leq K_u\cdot n^{-d/2}$, see \eqref{eq:bp18}.
This is also a new result, even for the specific models such as random interlacements, vacant set of random interlacements, and 
the level sets of the Gaussian free field. 
In the context of random interlacements, a bound close to optimal (with a correcting factor of a multiple logarithm) 
was obtained in \cite{PrShel}.
\item[(5)]
Theorem~\ref{thm:ip} implies that for any $u\in(a,b)$ and $\mathbb P^u[\cdot~|~0\in\set_\infty]$-almost every $\atom$, 
there exists $k_u = k_u(\atom)>0$ such that 
for all $n\geq 1$, $\mathbf P_{\atom,0}[X_{2n} = 0] \geq k_u\cdot n^{-d/2}$, see \cite[Remark~2.2]{BiskupPrescott}.
\end{itemize}
\end{remark}

Analogue of Theorem~\ref{thm:isoperimetric} has only been known before for independent Bernoulli percolation, 
see, e.g., \cite{Barlow,BM03,BBHK,MathieuRemy,Pete08}.
All these proofs rely crucially on a ``set counting argument'' and thus require exponential decay of probabilities of certain events. 
This is achieved by using Liggett-Schonmann-Stacey theorem \cite{LSS}. 
Such approach is quite restrictive and does not apply to models which cannot be compared with Bernoulli percolation on any scale, such as, 
for example, random interlacements. 
Our method is more robust and requires only minimal assumptions on the decay of probabilities of some events. 

\bigskip

We will now comment on the proof of Theorem~\ref{thm:isoperimetric}. 
As in all the proofs of isoperimetric inequalities for subsets of the infinite cluster of Bernoulli percolation, 
we set up a proper coarsening of $\set$ and then translate the given isoperimetric problem for large subsets of $\set$ into 
an isoperimetric problem on the coarsened lattice. 
Nevertheless, both the coarsening and the analysis of the coarsened lattice are very different from the ones used in existing approaches. 
The major difficulties, as already discussed, come from the presence of long-range correlations and the fact that 
the models cannot in general be compared with Bernoulli percolation, which rules out possibilities of using any Peierls-type argument. 

We partition the lattice $\Z^d$ into boxes $(x+ [0,L_0)^d)$, $x\in L_0\cdot\Z^d$, and subdivide the set of all boxes into {\it good} (very likely as $L_0\to\infty$) and {\it bad}. 
In the restriction of $\set$ to each of the good boxes, it is possible to identify uniquely a connected component of {\it largest volume}, which we call {\it special}, see Lemma~\ref{l:fromG0toZd}(a).  
Moreover, for any pair of adjacent good boxes, their special connected components are connected locally, see Lemma~\ref{l:fromG0toZd}(b). 
We emphasize that a good box may contain several connected components of large diameter, and in principle a connected component with the largest diameter may be different from the special component. 
This is the key difference of the coarsening procedure that we use from the ones used in the study of Bernoulli percolation. 
The main reason for doing this is that a box is good if two local events occur, one of which is increasing (there exists a large in volume connected component of $\set$ in the box) 
and the other decreasing (the cardinality of $\set$ in the box is not so big), see Definitions~\ref{def:Axu} and \ref{def:Bux}. 
Using \ppp{} to control correlations between such monotone events, 
we set up two multi-scale renormalizations with scales $L_n$ (one for increasing and one for decreasing events) to identify with high probability a well structured subset of good boxes. 
We call a $L_n$-box $n$-good if all the $L_{n-1}$-subboxes of this box, except for the ones contained in the union of at most two boxes of side length $r_{n-1}L_{n-1}$, 
are $(n-1)$-good. (Here $r_i$ is a sequence of positive integers growing to infinity, but much slower than $\frac{L_{i+1}}{L_i}$.) 
Every $L_n$-box is $n$-good with overwhelming probability. 
We are interested in the set of ($0$-)good boxes which are contained in $n$-good boxes for all $n\geq 1$. 
This set is obtained by a perforation of $\Z^d$ on multiple levels, and therefore has a well described structure. 
Indeed, from each $n$-good box, we delete two boxes of size length $r_{n-1}L_{n-1}$, from each of the remaining $L_{n-1}$-boxes (all of which are $(n-1)$-good), 
we delete two boxes of side length $r_{n-2}L_{n-2}$, and so on until we reach the level $0$. 
Moreover, if $r_i\ll \frac{L_{i+1}}{L_i}$, then the set of all deleted boxes (the complement of the good set) has a small volume. 
We should mention that such coarsening and renormalization have already been used before in \cite{DRS12,RS:Disordered} to study models with long-range correlations satisfying assumption \ppp{}. 

We need to further sparsen the obtained set of good boxes to make sure that it has good connectivity properties. 
This is done by a ``deterministic'' multi-level perforation of $\Z^d$, where from each $n$-good box, we 
delete yet at most one box of side length $r_{n-1}L_{n-1}$ depending on the location of two already deleted boxes of side length $r_{n-1}L_{n-1}$. 
For example, if two boxes of side length $r_{n-1}L_{n-1}$ are deleted near an edge of an $L_n$-box, then we delete 
another box of side length $r_{n-1}L_{n-1}$ at the edge, see Figures~\ref{fig:mathcalG} and \ref{fig:Gconstruction}.
During this discussion, we call the resulting (connected) set of good boxes {\it fat}. 
The fat set is not only connected in the lattice of $L_0$-boxes, but also 
its restriction to any lower dimensional sublattice $(x + \sum_{i=1}^j\Z\cdot e_i)$ is connected, where $x\in L_0\cdot \Z^d$, $2\leq j\leq d$, and $e_i\in\Z^d$ are pairwise orthogonal unit vectors, 
see Proposition~\ref{prop:Gproperties}. 
This property is crucially used in the proof of an isoperimetric inequality for subsets of fat set, but we will come to that.  

If the renormalization scales $L_n$ are growing fast enough, then the restriction of the fat set to the box $\ballZ(0,R)$ serves as a coarsening of the largest connected component $\mathcal C_R$, 
which also ensures uniqueness of $\mathcal C_R$. 
We would like to reduce the isoperimetric problem for large subsets of $\mathcal C_R$ to an isoperimetric inequality for large subsets of good boxes of the fat set. 
The main obstruction here is that our coarsening allows to identify a subset of $\mathcal C_R$ of large volume (union of special components of good boxes), but 
the remaining parts of $\mathcal C_R$ may contain long dangling ends with bad isoperimetric properties. 
We resolve this issue by two requirements on the set of configurations that we consider.
First of all, since we do not have any control of how $\set$ looks like in the ``deleted'' boxes of side length $r_{n-1}L_{n-1}$, 
we should at least make sure that each deleted region is not too big in comparison with $R^\ei$, the minimal size of sets which we consider. 
We require that on a level $s$ of renormalization such that $L_s^{3d^2}\leq R^\ei$ (see \eqref{def:s}), all the $L_s$-boxes intersecting $\ballZ(0,2R)$ are $s$-good, 
i.e., the biggest box that we ``delete'' has side length at most $r_{s-1}L_{s-1}$. 
Second, to get a partial control of connectivities in the dangling ends, 
we require that any $x,y\in\set_{L_s}\cap\ballZ(0,2R)$ such that $|x-y|_\infty<2L_s$ are connected in $\ballZ(x,4L_s)$. 
Configurations satisfying these assumptions form an event of high probability, and next we consider only configurations from this event. 

Given $A\subset\mathcal C_R$ such that $|A|\geq R^{\ei}$, we identify a subset $\setg$ of good boxes in the fat set for which $A$ intersects the special connected component. 
If $|\setg|$ is small, then we show that the boundary of $A$ is very large ($\geq c\cdot \frac{|A|}{L_s^d}$). 
The reason for this is that while most of the vertices $x$ of $A$ are not in special connected components of good boxes in the fat set, 
each of them is within distance at most $L_s$ from the fat set (all $L_s$-boxes in $\ballZ(0,2R)$ are $s$-good), and thus from $\mathcal C_R\setminus A$. 
The weak connectivity assumption then makes sure that there is an edge of $\partial_\set A$ in $\ballZ(x,4L_s)$, see Lemma~\ref{l:isopmain}. 
On the other hand, if $|\setg|$ is large, then we prove that it satisfies an isoperimetric inequality in the graph of good boxes, see Lemma~\ref{l:isopbaldG}. 
Noting that for any pair of good boxes from the boundary of $\setg$, their special connected components are locally connected, 
one of the special components intersects $A$ and the other does not, 
we obtain a lower bound on $|\partial_\set A|$ in terms of the size of the boundary of $\setg$, see \eqref{eq:isopmain2}.   

It remains to prove the isoperimetric inequality for large subsets $\mathbf A$ of good boxes of the fat set. 
We consider the set $\setgs$ of disjoint $L_s$-boxes such that at least half of $L_0$-boxes contained in it are from $\mathbf A$, see \eqref{eq:setgs}. 
Again, if $|\setgs|<C\cdot\frac{|\mathbf A|}{L_s^d}$, then the boundary of $\mathbf A$ is at least $c\cdot \frac{|\mathbf A|}{L_s^d}$. 
The interesting case is when $|\setgs|\geq C\cdot\frac{|\mathbf A|}{L_s^d}$. 
By the isoperimetric inequality in the lattice of $L_s$-boxes, the boundary of $\setgs$ has size $\geq c\cdot |\setgs|^{\frac{d-1}{d}} \geq c\cdot \frac{|\mathbf A|^{\frac{d-1}{d}}}{L_s^{d-1}}$.
We, roughly speaking, estimate the boundary of $\mathbf A$ from below 
by the part of its boundary restricted to (disjoint) $L_s$-boxes from the boundary of $\setgs$ and show that the restrictions to all the boxes are of size $\geq c\cdot L_s^{d-1}$. 
Thus the boundary of $\mathbf A$ contains $c\cdot \frac{|\mathbf A|^{\frac{d-1}{d}}}{L_s^{d-1}}$ disjoint pieces of size $c\cdot L_s^{d-1}$, and we are done. 

To be precise, for any pair of adjacent $L_s$-boxes from the boundary of $\setgs$, one box has large intersection with $\mathbf A$, and the other small. 
Therefore, the intersection of $\mathbf A$ with a box of side length $3L_s$ containing both $L_s$-boxes is comparable in size with its complement in this $3L_s$-box. 
We show that the boundary of $\mathbf A$ in any such $3L_s$-box is at least $c\cdot L_s^{d-1}$, see Lemma~\ref{l:isopa}. 
For this we prove a stronger statement that the restriction of $\mathbf A$ to $j$-dimensional subboxes ($2\leq j\leq d$) of a given $3L_s$-box 
which contain a non-trivial (bounded away from $0$ and $1$) density of vertices from $\mathbf A$
satisfies an isoperimetric inequality in those subboxes, i.e., its boundary in the graph of good boxes in the $j$-dimensional subbox has size $\geq c_j\cdot L_s^{j-1}$, 
see Lemma~\ref{l:isopaj}. 
The last statement is proved by induction on $j$. 

In the case $j=2$, we first reduce the problem to {\it connected} sets with complement consisting of large connected components, see \eqref{eq:isopaconnected}. 
Then by using the precise construction of the fat set (from each $n$-good box we delete at most $3$ boxes of side length $r_{n-1}L_{n-1}$), 
we show that the boundary of the set in the graph of good boxes has almost the same size as 
the boundary of the set in $L_0\cdot \Z^2$, i.e., the part of the boundary of the set which touches some ``deleted'' boxes is small, see \eqref{eq:relationboundaries}. 
In the case $j\geq 3$, we use a dimension reduction argument. 
We first consider $(j-1)$-dimensional subcubes (slices) of a given $j$-dimensional subcube which are stacked along one particular coordinate direction. 
If there is a positive fraction of slices which have large intersections with $\mathbf A$ and its complement, then we use induction assumption for these slices. 
Otherwise, we conclude that there are two large disjoint subsets of slices, those that contain many vertices from $\mathbf A$ and very few from its complement, 
and those that contain few vertices from $\mathbf A$ and many from its complement (overcrowded and undercrowded slices). 
We then consider two-dimensional slices that intersect all these $(j-1)$-dimensional slices, see Figure~\ref{fig:slices}.
Most of them will have large intersection with $\mathbf A$ as well as with its complement. 
We conclude by using the isoperimetric inequality in each of these two dimensional slices (case $j=2$).

\bigskip

\subsection{Examples}

It is well known that classical supercritical Bernoulli percolation satisfies all the requirements \p{} -- \ppp{} and \s{} -- \sss{}. 
The main focus of this paper is on models with long range correlations, especially the ones that cannot be studied by comparison with Bernoulli percolation 
on any scale. 
The following models with {\it polynomial decay of correlations} are known to satisfy all the requirements \p{} -- \ppp{} and \s{} -- \sss{}, 
see \cite[Section~2]{DRS12}:
\begin{itemize}\itemsep0pt
\item[(a)]
{\it random interlacements} at any level $u>0$ (see \cite{SznitmanAM});
\item[(b)]
{\it vacant set of random interlacements} at level $u$ (see \cite{SznitmanAM,SidoraviciusSznitman_RI}) 
in the (non-empty) regime of the so-called local uniqueness;
\item[(c)]
{\it level sets of the Gaussian free field} (see \cite{LS86,RodSz}) 
also in the (non-empty) regime of local uniqueness;
\end{itemize}
The regime of local uniqueness is basically described by those values of $u$ for which \s{} is fulfilled. 
It was shown that the regime of local uniqueness is non-empty for the vacant set of random interlacements in \cite[Theorem~1.1]{DRS}, 
and for the level sets of the Gaussian free field in \cite[Theorem~2.6]{DRS12}. 
In the case of Bernoulli percolation, it is well known that the regime of local uniqueness coincides with the whole supercritical phase, 
and, based on this, it is believed that the same is true for the models in (b) and (c). 
It was proved in \cite{DRS12} that both models satisfy all the requirements except for \s{} in the {\it whole} supercritical regime. 
Thus, in order to extend the results of this paper to the whole supercritical phase in models (b) and (c), it suffices to check that 
\s{} is satisfied for all supercritical values of $u$. Currently, this remains an open problem.

\subsection{Structure of the paper}
In Section~\ref{sec:renormalization}, we recall the renormalization scheme of \cite{DRS12}. 
Theorem~\ref{thm:isoperimetric} is proved in Section~\ref{sec:isoperimetric} (see the beginning of the section for 
a detailed description of its content). 
In Section~\ref{sec:ipa} we state a quenched invariance principle for random walk on the infinite percolation cluster of 
a random subset of $\Z^d$ satisfying a list of general conditions. 
We show that these conditions are implied by \p{} -- \ppp{} and \s{} -- \sss{} in Section~\ref{sec:ip}. 
We discuss possible weakenings of assumption \p{} in Section~\ref{sec:ergodicityremarks}. 
Last, in Section~\ref{sec:ipaproof} we give a sketch proof of the general quenched invariance principle stated in Section~\ref{sec:ipa}; 
this is a routine adaptation of techniques present in the literature.

\bigskip

Finally, let us make a convention about constants. 
As already mentioned, we omit the dependence of various constants on $a$, $b$, and $d$ from the notation.
Dependence on other parameters is reflected in the notation, for example, as $c(u,\ei)$.

\section{Renormalization}
\label{sec:renormalization}

In this section we recall the renormalization scheme from \cite[Sections~3-5]{DRS12}. 
(Some ideas are already present in \cite{RS:Disordered} in the context of random interlacements and its vacant set.) 
The goal is to define a coarsening of $\set$ using monotone events from Definitions~\ref{def:Axu} and \ref{def:Bux}, and 
identify its connectivity patterns using a multi-scale renormalization with scales $L_n$, see \eqref{def:scales}, \eqref{def:Auxk}, \eqref{def:Buxk}, and Lemma~\ref{l:Guxk}. 
The key notion is of $k$-good vertices (boxes), see Definition~\ref{def:good} and Lemma~\ref{l:Guxk}. 
The main property of a $0$-good box is that it contains a unique connected component of $\set$ with largest volume, see Lemma~\ref{l:fromG0toZd}(a), 
and for any pair of adjacent good boxes, their unique largest connected components are connected locally, see Lemma~\ref{l:fromG0toZd}(b). 
For $k\geq 1$, the $k$-good box is defined recursively such that  
all its $(k-1)$-bad subboxes are contained in at most two subboxes of side length $r_{k-1}L_{k-1}$, where $r_{k-1}L_{k-1}\ll L_k$, see Definition~\ref{def:good}. 

\medskip  

Let $\scexp = \lceil 1/\epsP \rceil$, where $\epsP$ is defined in \ppp{}. Let $r_0$, $l_0$, and $L_0$ be positive integers. 
(Later in the proofs, we will assume that these integers are sufficiently large, and that the ratio $\frac{r_0}{l_0}$ is sufficiently small, 
see the discussion before Section~\ref{sec:eventH}.)
Consider the sequences of positive integers
\begin{equation}\label{def:scales}
l_k = l_0\cdot 4^{k^\scexp},\qquad r_k = r_0\cdot 2^{k^\scexp},\qquad L_k = l_{k-1}\cdot L_{k-1}, \quad k\geq 1.
\end{equation}
For $k\geq 0$, we introduce the renormalized lattice graph $\GG_k$ by
\begin{equation*}
\GG_k = L_k \Z^d = \{L_kx ~:~ x\in\Z^d\} ,\
\end{equation*}
with edges between any pair of $\ell^1$-nearest neighbor vertices of $\GG_k$. 
\begin{definition}\label{def:Axu}
For $x\in \GG_0$ and $u\in(a,b)$, let $A^u_x\in\mathcal F$ be the event that
\begin{itemize}\itemsep0pt
\item[(a)]
for each $e\in\{0,1\}^d$, the set $\set_{L_0}\cap(x+eL_0 + [0,L_0)^d)$
contains a connected component with at least $\frac 34 \eta(u) L_0^d$ vertices,
\item[(b)]
all of these $2^d$ components are connected in 
$\set \cap(x+[0,2L_0)^d)$.
\end{itemize}
\end{definition}
For $u\in(a,b)$ and $x\in \GG_0$, let 
$\overline A^u_{x,0}$ be the compelement of $A^u_{x}$, 
and for $u\in(a,b)$, $k\geq 1$, and $x\in\GG_k$ define inductively  
\begin{equation}\label{def:Auxk}
\overline A^u_{x,k} = 
\bigcup_{\begin{array}{c}\scriptscriptstyle{x_1,x_2\in \GG_{k-1}\cap(x + [0,L_k)^d)} \\ \scriptscriptstyle{|x_1-x_2|_\infty \geq r_{k-1} \cdot L_{k-1}}\end{array}}
  \overline A^u_{x_1,k-1} \cap \overline A^u_{x_2,k-1}   \; .\
\end{equation}

\begin{definition}\label{def:Bux}
For $x\in \GG_0$ and $u\in(a,b)$, let $B^u_x\in\mathcal F$ be the event that
for all $e\in\{0,1\}^d$, 
\[
\left|\set_{L_0}\cap(x+eL_0 + [0,L_0)^d)\right| \leq \frac54\eta(u) L_0^d .\
\]
\end{definition}
For $u\in(a,b)$ and $x\in \GG_0$, let 
$\overline B^u_{x,0}$ be the complement of $B^u_{x}$, 
and for $u\in(a,b)$, $k\geq 1$, and $x\in\GG_k$ define inductively
\begin{equation}\label{def:Buxk}
\overline B^u_{x,k} = 
\bigcup_{\begin{array}{c}\scriptscriptstyle{x_1,x_2\in \GG_{k-1}\cap(x + [0,L_k)^d)} \\ \scriptscriptstyle{|x_1-x_2|_\infty \geq r_{k-1} \cdot L_{k-1}}\end{array}}
  \overline B^u_{x_1,k-1} \cap \overline B^u_{x_2,k-1}   \; .\
\end{equation}

\begin{definition}\label{def:good}
Let $u\in(a,b)$. For $k\geq 0$, we say that $x\in\GG_k$ is $k$-{\it bad} if the event 
$\overline A^u_{x,k}\cup \overline B^u_{x,k}$
occurs. Otherwise, we say that $x$ is $k$-{\it good}. 
Note that $x\in\GG_0$ is $0$-good, if the event $A^u_x\cap B^u_x$ occurs.
\end{definition}
The following result is \cite[Lemmas~4.2 and 4.4]{DRS12}.
\begin{lemma}\label{l:Guxk} 
Assume that the measures $\mathbb P^u$, $u\in(a,b)$, satisfy conditions \p{} -- \ppp{} and \s{} -- \sss{}.
Let $l_k$, $r_k$, and $L_k$ be defined as in \eqref{def:scales}. 
For each $u\in(a,b)$, there exist $C = C(u)<\infty$ and $C' = C'(u,l_0)<\infty$ such that for all 
$l_0,r_0\geq C$, $L_0\geq C'$, and $k\geq 0$, 
\[
\mathbb P^u\left[0\mbox{ is }k\mbox{-bad}\right] \leq 2\cdot 2^{-2^k} .\
\]
\end{lemma}
\begin{remark}\label{rem:conditions}
The proof of Lemma~\ref{l:Guxk} crucially relies on conditions \pp{}, \ppp{}, and \sss{}, see \cite{DRS12}. 
This is the only place in the proof of Theorem~\ref{thm:isoperimetric} where we use these conditions. 
In the proof of Theorem~\ref{thm:ip} we use these conditions also to prove \eqref{eq:kbad:new}, 
which is a slightly stronger version of Lemma~\ref{l:Guxk} 
(and its proof is essentially the same as the proof of Lemma~\ref{l:Guxk}).
\end{remark}
The next result is \cite[Lemma~5.2]{DRS12}.
\begin{lemma}\label{l:fromG0toZd}
Let $x,y\in\GG_0$ be nearest neighbors in $\GG_0$ such that both are $0$-good. Then 
\begin{itemize}\itemsep0pt
 \item[(a)]
each of the graphs $\set_{L_0}\cap(z + [0,L_0)^d)$, with $z\in \{x,y\}$, 
contains the unique connected component $\mathcal C_z$ with at least $\frac34 \eta(u) L_0^d$ vertices, 
\item[(b)] 
$\mathcal C_x$ and $\mathcal C_y$ are connected in the graph $\set\cap ((x+[0,2L_0)^d)\cup(y+[0,2L_0)^d))$. 
\end{itemize}
\end{lemma}

\section{Proof of Theorem~\ref{thm:isoperimetric}}\label{sec:isoperimetric}

The proof of Theorem~\ref{thm:isoperimetric} consists of a probabilistic part, 
in which we impose some restrictions on the set of allowed configurations (see Defintion~\ref{def:mathcalh}) and 
estimate the probability of the resulting event $\mathcal H$ (see \eqref{eq:eventH:proba}), 
and a deterministic part, in which we prove the isoperimetric inequality for subsets of the largest connected component of $\set\cap\ballZ(0,R)$ 
for each configuration satisfying the a priori restrictions. 

We identify two special levels of the renormalization, $s$ defined in \eqref{def:s}, and $r = \lfloor \frac{s}{2}\rfloor$. 
They are defined so that on the one hand $L_s^d\ll R^\ei$ (which means that ``deleted'' subboxes are rather small), and on the other hand, 
the probability that a vertex is $r$-bad is still very small in $R$.
We use these scales to define the event $\mathcal H$, which consists of configurations for which all the vertices from $\ballZ(0,3R)\cap (L_r\cdot \Z^d)$ are $r$-good, and 
any $x,y\in \set_{L_s}\cap\ballZ(0,2R)$ such that $|x-y|_\infty\leq 2L_s$ are connected in $\ballZ(x,4L_s)$, see Defintion~\ref{def:mathcalh}.
Using Lemma~\ref{l:Guxk} and assumption \s{} we show that the probability of $\mathcal H$ is close to $1$, see \eqref{eq:eventH:proba}. 
(The event $\mathcal H$ depends on $d$, $u$, $\ei$, and $R$, but we do not reflect this in the notation.)

Next, using combinatorics we show that any configuration from $\mathcal H$ belongs to the event in \eqref{eq:isoperimetric}. 
This is done in several steps. 
First, using the notion of $k$-good vertices from Definition~\ref{def:good}, we identify for each configuration in $\mathcal H$ a well structured connected (in $\GG_0 = L_0\cdot \Z^d$) 
subset $\mathbf G$ of $0$-good vertices in $\GG_0\cap\ballZ(0,2R)$, obtained from $\GG_0\cap\ballZ(0,2R)$ 
by a certain multi-scale perforation procedure. 
The set $\mathbf G$ consists roughly of those $0$-good vertices in $\GG_0\cap\ballZ(0,2R)$ which are contained only in $j$-good boxes for all $1\leq j\leq r$. 
(We need to sparsen this set a bit more in order to obtain the actual set $\mathbf G$ with the desired connectivity properties.)
This set is well connected, ubiquitous in $\GG_0\cap\ballZ(0,2R)$, 
and has almost the same volume as $\GG_0\cap\ballZ(0,2R)$, see Proposition~\ref{prop:Gproperties}. 
A crucial step in the proof is a reduction of the initial isoperimetric problem 
for subsets of the largest cluster of $\set$ in $\ballZ(0,R)$ to an isoperimetric problem for large enough subsets of $\mathbf G$, 
see Lemmas~\ref{l:isopmain} and \ref{l:isopbaldG}.
The rest of the proof is then about isoperimetric properties of large subsets $\mathbf A$ of $\mathbf G$, see Lemma~\ref{l:isopbaldG}. 
If $\mathbf A$ is sparse, then its boundary is almost comparable with the volume of $\mathbf A$. 
The most delicate case is when $\mathbf A$ is localized, since in this case its boundary may be much smaller than its volume. 
In this case, we estimate the boundary of $\mathbf A$ locally in each of the boxes of side length $3L_s$ which are densely occupied by $\mathbf A$ and by its complement, see Lemma~\ref{l:isopa}. 
We show that in each of such boxes, the boundary of $\mathbf A$ is at least $c\cdot L_s^{d-1}$. 
For that we prove a stronger statement that the restriction of $\mathbf A$ to (many) $j$-dimensional hyperplanes ($2\leq j\leq d$) intersecting the given box of side length $3L_s$ 
has boundary $\geq c_j\cdot L_s^{j-1}$, see Lemma~\ref{l:isopaj}. This proof is by induction on $j$.

\medskip

In the proof of Theorem~\ref{thm:isoperimetric} we will work with the scales $L_k$, $l_k$, and $r_k$ defined in \eqref{def:scales}. 
Throughout the proof we take $r_0$, $l_0$, and $L_0$ satisfying Lemma~\ref{l:Guxk}. 
We need to adjust these parameters further in the proof as follows:
\begin{itemize}\itemsep0pt
\item
in the construction of $\mathbf G$, we assume that $4r_0<l_0$, 
which is essential for the connectedness of $\mathbf G$. 
\item
in showing that the largest (in volume) connected component of $\set\cap\ballZ(0,2R)$ is uniquely defined, we assume that 
$L_0$ is large enough and $\frac{r_0}{l_0}$ is small enough (both depending on $u$) to satisfy \eqref{eq:L_0:condition}. 
\item
we choose $\frac{r_0}{l_0}$ small enough to satisfy Lemmas~\ref{l:isopmain} and \ref{l:isopbaldG}.
\end{itemize}
Most of the conditions on the smallness of $\frac{r_0}{l_0}$ are formulated in terms of the closeness to $1$ of 
\begin{equation}\label{eq:fj}
f_j\left(\frac{r_0}{l_0}\right) = \prod_{i=0}^\infty\left(1 - 3\cdot \left(\frac{r_i}{l_i}\right)^j\right)  .\
\end{equation}
(See, \eqref{eq:r_0l_0:isopmain}, \eqref{eq:fj:isopaj}, and \eqref{eq:r_0l_0isopaj3}.)
The only exception is \eqref{eq:r_0l_0:isopaj2}. 

The reader may notice that in Lemma~\ref{l:isopaj} we choose the ratio $\frac{r_0}{l_0}$ small enough depending on a parameter $\epsilon$, 
see \eqref{eq:fj:isopaj} and \eqref{eq:r_0l_0isopaj3}. 
This is fine, since in the end we only use Lemma~\ref{l:isopaj} for a specific choice of $\epsilon = \frac{1}{2\cdot 3^d}$.

\subsection{The event $\mathcal H$ and its probability}\label{sec:eventH}

In this section we define the event $\mathcal H$ containing all the restrictions on the set of allowed configurations (see Definition~\ref{def:mathcalh}), 
and show that it has probability close to $1$ (see \eqref{eq:eventH:proba}). 
The event $\mathcal H$ depends on $d$, $u$, $\ei$, and $R$, but we do not reflect this in the notation.

\medskip

Recall the definition of $\ei$ from the statement of Theorem~\ref{thm:isoperimetric}, and 
note that it suffices to assume that $\ei\in(0,1)$. 

Let $s$ be the largest integer such that $L_s^{3d^2} \leq R^{\ei}$, i.e.,
\begin{equation}\label{def:s}
s=\max \left\{ \, s' \; : \;   L_{s'}^{3d^2} \leq R^{\ei} \right\} .\
\end{equation}
We assume that $R\geq L_0^{3d^2/\ei}$, so that $s$ is well-defined. 

Let $r = \lfloor \frac{s}{2}\rfloor$. By \eqref{def:scales}, 
\begin{equation}\label{eq:scalesrs}
L_r^2\leq L_0\cdot L_s .\
\end{equation}

\begin{remark}
\begin{itemize}\itemsep0pt
\item[(1)]
We need to choose the power of $L_{s'}$ in \eqref{def:s} large enough, 
only in order to deduce Theorem~\ref{thm:isoperimetric} from Lemmas~\ref{l:isopmain} and \ref{l:isopbaldG}. 
In fact, any exponent bigger than $3d^2$ would also do, with a suitable change of constants in \eqref{eq:Ls:lowerbound} and \eqref{eq:bounds}. 
\item[(2)]
Property \eqref{eq:scalesrs} will be crucial in the proof of the isoperimetric inequality for two dimensional slices, see 
Lemma~\ref{l:isopaj} and the proof of \eqref{eq:isopaconnected}. 
\end{itemize}
\end{remark}
By \eqref{def:scales} and \eqref{def:s}, for all $R\geq L_0^{3d^2/\ei}$, 
\[
R^{\frac{\ei}{3d^2}} \leq L_{s+1} = l_s\cdot L_s \leq l_0\cdot 4 \cdot (L_s)^{1+2^\scexp},
\]
which implies that 
\begin{equation}\label{eq:Ls:lowerbound}
L_s \geq \frac{1}{4l_0} R^{\frac{\ei}{3d^2(1+2^\scexp)}} .\
\end{equation}
From \eqref{def:scales} and \eqref{eq:Ls:lowerbound} we deduce that there exists $c = c(\ei, \scexp, l_0, L_0)>0$ such that for all $R\geq L_0^{3d^2/\ei}$, 
\begin{equation}\label{eq:bounds}
s\geq c\cdot(\log R)^{\frac{1}{1+\scexp}} - 1 .\
\end{equation}

\bigskip

Next we define the event $\mathcal H$.
\begin{definition}\label{def:mathcalh}
Let $u\in(a,b)$. Consider the event $\mathcal H$ that 
\begin{itemize}\itemsep0pt
\item[(a)]
each $z\in\GG_r\cap\ballZ(0,3R)$ is $r$-good, 
\item[(b)]
for any $z\in\ballZ(0,2R)$ and $x,y\in\set_{L_s}\cap(z+[-2L_s,2L_s)^d)$, 
$x$ is connected to $y$ in $\set\cap(z+[-4L_s,4L_s)^d)$. 
\end{itemize}
\end{definition}
\begin{remark}\label{rem:sgood}
Property (a) in the definition of $\mathcal H$ implies the weaker version (a') stating that each $z\in\GG_s\cap\ballZ(0,2R)$ is $s$-good. 
Most of the arguments in the proof of Theorem~\ref{thm:isoperimetric} would go through if we used (a') instead of (a) in the definition of $\mathcal H$. 
The only point where we essentially use (a) is in the proof of the two dimesional case, 
see Lemma~\ref{l:isopaj} and the proof of \eqref{eq:isopaconnected}. 
\end{remark}
By Definition~\ref{def:good}, \s{}, and Lemma \ref{l:Guxk}, there exists $C = C(u)<\infty$ such that 
\[
\mathbb P^u\left[\mathcal H^c \right]\leq 
|\GG_r\cap\ballZ(0,3R)|\cdot 2\cdot 2^{-2^r} + |\ballZ(0,2R)|\cdot C\cdot e^{-\funcS(u,2L_s)} ~.\
\]
Using \eqref{eq:funcS}, \eqref{eq:Ls:lowerbound}, and \eqref{eq:bounds},  
we deduce that there exist $c' = c'(u,\ei,\scexp)>0$ and $C' = C'(u,\ei, \scexp, l_0, L_0)<\infty$ such that for all $R\geq C'$, 
\begin{equation}\label{eq:s:bounds}
2^r \geq (\log R)^{1+\constS} \quad\mbox{and}\quad \funcS(u,2L_s) \geq c'(\log R)^{1+\constS} ~.\
\end{equation}
Note that also $|\GG_r\cap\ballZ(0,3R)| \leq |\ballZ(0,3R)|\leq (6R+1)^d$.
Therefore, there exist $c=c(u,\ei,\scexp)>0$ and $C = C(u,\ei,\scexp,l_0,L_0)<\infty$ such that for all $R\geq C$,  
\begin{equation}\label{eq:eventH:proba}
\mathbb P^u\left[\mathcal H \right]\geq 1 - Ce^{-c(\log R)^{1+\constS}} .\
\end{equation}
In the remaining part of the proof, we will show that 
each configuration from $\mathcal H$ also belongs to the event in \eqref{eq:isoperimetric}. 
Together with \eqref{eq:eventH:proba} this will imply Theorem~\ref{thm:isoperimetric}.
\[
\text{From now on we assume that $\mathcal H$ occurs.}
\]

\subsection{Construction of $\mathbf G$}

In this section we construct the subset $\mathbf G$ of $0$-good vertices in $\GG_0\cap\ballZ(0,2R)$ 
with the property that for every $z_0\in\mathbf G$ and each of the boxes $(z_j+[0,L_j)^d)$, $z_j\in\GG_j$, $0<j\leq s$, containing $z_0$, the vertex $z_j$ is $j$-good, 
and also that the set $\mathbf G$ exhibits good properties of density and connectedness, see Proposition~\ref{prop:Gproperties}.
The construction is done recursively by going down through the renormalization levels and using Definition~\ref{def:good}. 
We assume throughout the construction that $4r_0<l_0$ (which implies that $4r_i<l_i$ for all $i$). This is essential 
for the connectedness of the sets below.

\medskip

Let $\mathcal G_s=\GG_s\cap\ballZ(0,2R-2L_s)$. 
By the definition of $\mathcal H$ and Remark~\ref{rem:sgood}, all $z_s\in\mathcal G_s$ are $s$-good. 
Also note that $\cup_{z_{s}\in\mathcal G_{s}}(z_{s} + [0,L_{s})^d)\subset \ballZ(0,2R-L_s)$.

For $r\leq i<s$, let $\mathcal G_i = \cup_{z_{i+1}\in\mathcal G_{i+1}}(\GG_i\cap(z_{i+1} + [0,L_{i+1})^d)$. 
By the definition of $\mathcal H$, all $z_r\in\mathcal G_r$ are $r$-good.

Next we take $0<i\leq r$ and assume that $\mathcal G_i\subset \GG_i$ is defined so that
\begin{itemize}\itemsep0pt
\item
all $z_{i}\in\mathcal G_{i}$ are $i$-good,
\item
for any $z_s\in\mathcal G_s$, the set $\mathcal G_i\cap(z_s + [0,L_s)^d)$ is connected in $\GG_i$ and 
\[
\left|\mathcal G_{i}\cap(z_s + [0,L_s)^d)\right| \geq \left|\mathcal G_{i+1}\cap(z_s + [0,L_s)^d)\right|\cdot l_{i}^d\cdot \left(1 - 3\left(\frac{r_{i}}{l_{i}}\right)^d\right),
\]
\item 
for any $z_s,\widetilde z_s\in\mathcal G_s$ with $|z_s - \widetilde z_s|_1 = L_s$, 
the set $\mathcal G_i\cap((z_s + [0,L_s)^d)\cup(\widetilde z_s + [0,L_s)^d))$ is connected in $\GG_i$,
\item
for any $z_s\in\mathcal G_s$, $x_i\in\GG_i\cap(z_s + [0,L_s)^d)$, and two orthogonal $e,e'\in\Z^d$ with $|e|_1 = |e'|_1 = 1$,
the two dimensional slice $\mathcal G_{i}\cap(x_i + \Z\cdot e + \Z\cdot e')\cap(z_s + [0,L_s)^d)$ 
is connected in $\GG_i\cap(x_i + \Z\cdot e + \Z\cdot e')$, and 
\begin{multline*}
\left|\mathcal G_i\cap(x_i + \Z\cdot e + \Z\cdot e')\cap(z_s + [0,L_s)^d)\right|\\ \geq 
\left|\mathcal G_{i+1}\cap(x_{i+1} + \Z\cdot e + \Z\cdot e')\cap(z_s + [0,L_s)^d)\right|
\cdot l_{i}^2\cdot\left(1 - 3\left(\frac{r_{i}}{l_{i}}\right)^2\right),
\end{multline*}
where $x_{i+1}\in\GG_{i+1}$ satisfies $x_i\in(x_{i+1} + [0,L_{i+1})^d)$,
\item
for any $z_s,\widetilde z_s\in\mathcal G_s$ with $|z_s - \widetilde z_s|_1 = L_s$, 
$x_i\in\GG_i\cap(z_s + [0,L_s)^d)$, $e = \frac{z_s - \widetilde z_s}{L_s}$, and $e'\in\Z^d$ orthogonal to $e$ such that $|e'|_1 = 1$, 
the two dimensional slice $\mathcal G_{i}\cap(x_i + \Z\cdot e + \Z\cdot e')\cap((z_s + [0,L_s)^d)\cup(\widetilde z_s + [0,L_s)^d))$ 
is connected in $\GG_i\cap(x_i + \Z\cdot e + \Z\cdot e')$.
\end{itemize}
We now define $\mathcal G_{i-1}\subset\GG_{i-1}$ which satisfies the same properties as $\mathcal G_i$ with $i$ replaced everywhere by $(i-1)$. 
By Definition~\ref{def:good}, for each $z_i\in\mathcal G_i$, there exist $a_{z_i},b_{z_i}\in\GG_{i-1}\cap(z_i + [0,L_i)^d)$ such that 
all the vertices in $\mathcal G_{z_i}' = (\GG_{i-1}\cap(z_i + [0,L_i)^d))\setminus((a_{z_i}+[0,r_{i-1}L_{i-1})^d)\cup(b_{z_i}+[0,r_{i-1}L_{i-1})^d))$ are $(i-1)$-good.
Note that the set $\mathcal G_{z_i}'$ is connected in $\GG_{i-1}$ if $d\geq 3$ (and $4r_{i-1}<l_{i-1}$), 
but not necessarily connected if $d=2$. 
\begin{figure}
\centering
\includegraphics[width=0.9\textwidth]{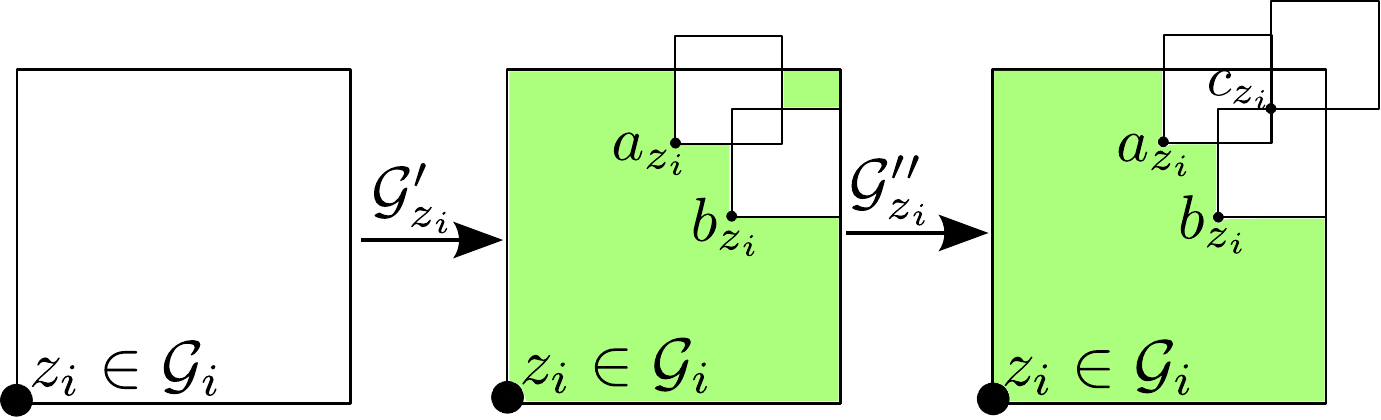}
\caption{The colored region on the middle picture corresponds to $\mathcal G_{z_i}'$. 
Its restriction to a two dimensional hyperplane is not generally connected, as illustrated here. 
The colored region on the last picture corresponds to $\mathcal G_{z_i}''$. 
The restriction of $\mathcal G_{z_i}''$ to any two dimensional hyperplane is connected.  
The three small boxes on these figures are not drawn to the actual scale. (Mind that we assume that $4r_{i-1}<l_{i-1}$.)}
\label{fig:mathcalG}
\end{figure}
However, for each $z_i\in\mathcal G_i$, there exists $c_{z_i}\in\GG_{i-1}\cap(z_i + [0,L_i)^d)$ (see Figure~\ref{fig:mathcalG}) such that 
the sets $\mathcal G_{z_i}'' = \mathcal G_{z_i}'\setminus(c_{z_i}+[0,r_{i-1}L_{i-1})^d)$, $z_i\in\mathcal G_i$,  
satisfy the following properties:
\begin{itemize}
\item
for any $z_i\in\mathcal G_i$, all $z_{i-1}\in\mathcal G_{z_i}''$ are $(i-1)$-good,
\item
for any $z_i\in\mathcal G_i$, the set $\mathcal G_{z_i}''$ is connected in $\GG_{i-1}$ and 
$|\mathcal G_{z_i}''| \geq l_{i-1}^d\cdot\left(1 - 3\left(\frac{r_{i-1}}{l_{i-1}}\right)^d\right)$,
\item
for any $z_i,\widetilde z_i\in\mathcal G_i$ with $|z_i - \widetilde z_i|_1 = L_i$, 
the set $\mathcal G_{z_i}''\cup\mathcal G_{\widetilde z_i}''$ is connected in $\GG_{i-1}$,
\item
for any $z_i\in\mathcal G_i$, $x_{i-1}\in\GG_{i-1}\cap(z_i + [0,L_i)^d)$, and two orthogonal $e,e'\in\Z^d$ with $|e|_1 = |e'|_1 = 1$,
the two dimensional slice $\mathcal G_{z_i}''\cap(x_{i-1} + \Z\cdot e + \Z\cdot e')$ 
is connected in $\GG_{i-1}\cap(x_{i-1} + \Z\cdot e + \Z\cdot e')$, and 
$\left|\mathcal G_{z_i}''\cap(x_{i-1} + \Z\cdot e + \Z\cdot e')\right| \geq l_{i-1}^2\cdot\left(1 - 3\left(\frac{r_{i-1}}{l_{i-1}}\right)^2\right)$,
\item
for any $z_i,\widetilde z_i\in\mathcal G_i$ with $|z_i - \widetilde z_i|_1 = L_i$, 
$x_{i-1}\in\GG_{i-1}\cap(z_i + [0,L_i)^d)$, $e = \frac{z_i - \widetilde z_i}{L_i}$, and $e'\in\Z^d$ orthogonal to $e$ such that $|e'|_1 = 1$, 
the two dimensional slice $(\mathcal G_{z_i}''\cup\mathcal G_{\widetilde z_i}'')\cap(x_{i-1} + \Z\cdot e + \Z\cdot e')$ 
is connected in $\GG_{i-1}\cap(x_{i-1} + \Z\cdot e + \Z\cdot e')$.
\end{itemize}
We define $\mathcal G_{i-1} = \cup_{z_i\in\mathcal G_i}\mathcal G_{z_i}''\subset\GG_{i-1}$. 
From the above properties of $(\mathcal G_{z_i}'')_{z_i\in\mathcal G_i}$, one can see that
$\mathcal G_{i-1}$ satisfies the same properties as $\mathcal G_i$ with $i$ replaced everywhere by $(i-1)$.
\begin{figure}[t]
\centering
\includegraphics[width=0.5\textwidth]{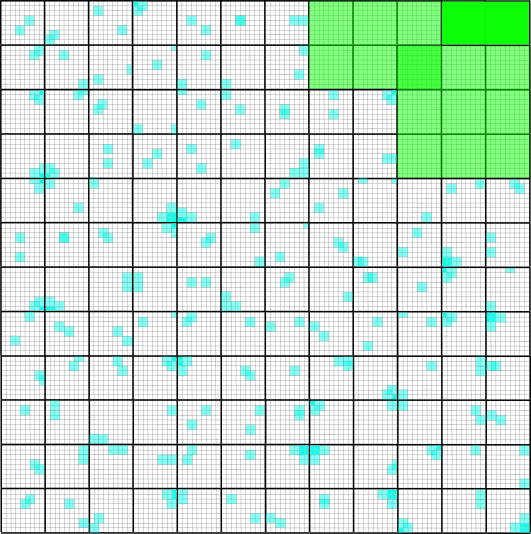}
\caption{This is an illustration of $\mathcal G_1$ and $\mathcal G_0 = \mathbf G$ in two dimensions in a box $(z_2 + [0,L_2)^2)$, for some $z_2\in\mathcal G_2$. 
Here $l_1 = 12$, $r_1 = 3$, $l_0 = 9$, and $r_0 = 2$. 
The box $(z_2 + [0,L_2)^2)$ consists of $12\times 12$ boxes of size $L_1$. 
The left-bottom corners of those boxes (of size $L_1$) which are not colored belong to $\mathcal G_1\cap(z_2 + [0,L_2)^2)$, 
and the left-bottom corners of small white boxes (of size $L_0$) belong to $\mathcal G_0\cap(z_2 + [0,L_2)^2)$.
Since $4r_i<l_i$ for $i\in\{0,1\}$, the resulting set is connected in $\GG_0$.} 
\label{fig:Gconstruction}
\end{figure}

\medskip

The outcome of such a recursive procedure is the set $\mathcal G_0\subset\cup_{z_s\in\mathcal G_s}(z_s + [0,L_s)^d)\subset \GG_0\cap\ballZ(0,2R-L_s)$ 
on the $0$-level of the renormalization. 
We denote it by $\mathbf G$. (See an illustration of part of $\mathbf G$ on Figure~\ref{fig:Gconstruction}.)
Note that $\mathbf G$ satisfies all the properties of $\mathcal G_i$ listed above with $i$ replaced everywhere by $0$. 
For the ease of references, we summarize most of the properties of $\mathbf G$ used in the proof of Theorem~\ref{thm:isoperimetric} in Proposition~\ref{prop:Gproperties}
(see also Remark~\ref{rem:Gexception}). These properties follow from the construction.

\begin{proposition}\label{prop:Gproperties}
The set $\mathbf G\subset \GG_0\cap\ballZ(0,2R-L_s)$ constructed above satisfies the following properties (with $f_j$ defined in \eqref{eq:fj}):
\begin{itemize}\itemsep0pt
\item[(a)]
any $z\in\mathbf G$ is $0$-good,
\item[(b)]
for any $x_s\in\GG_s\cap\ballZ(0,2R-2L_s)$, the set $\mathbf G\cap(x_s + [0,L_s)^d)$ is connected in $\GG_0$ and 
$\left|\mathbf G\cap(x_s + [0,L_s)^d)\right|\geq \left(\frac{L_s}{L_0}\right)^d\cdot f_d\left(\frac{r_0}{l_0}\right)$,
\item[(c)]
for any $x_s\in\GG_s\cap\ballZ(0,2R-3L_s)$, $x_0\in\GG_0\cap(x_s+[-L_s,2L_s)^d)$, $2\leq j\leq d$, 
and pairwise orthogonal $e_1,\dots,e_j\in\Z^d$ with $|e_i|_1 = 1$, 
the $j$-dimensional slice $\mathbf G\cap(x_s + [-L_s,2L_s)^d)\cap(x_0 + \sum_{i=1}^j\Z\cdot e_i)$ is connected in $\GG_0$ and 
$\left|\mathbf G\cap(x_s + [-L_s,2L_s)^d)\cap(x_0 + \sum_{i=1}^j\Z\cdot e_i)\right|\geq \left(\frac{3L_s}{L_0}\right)^j\cdot f_j\left(\frac{r_0}{l_0}\right)$.
\end{itemize}
\end{proposition}
\begin{remark}\label{rem:Gexception}
Most part of the proof of Theorem~\ref{thm:isoperimetric} relies on the properties of $\mathbf G$ listed in Proposition~\ref{prop:Gproperties}, 
and not on the specifics of the construction of $\mathbf G$. 
The only exception is the proof of the isoperimetric inequality for two dimensional slices, 
where we need to use the definition of {\it all} $\mathcal G_i$'s,
see Lemma~\ref{l:isopaj} and especially the proof of \eqref{eq:relationboundaries}. 
\end{remark}

\medskip

In what follows, we will use ordinary font to denote subsets of $\set$ (e.g., $A$ and $\setd$), 
bold font for subsets of $\mathbf G$ (e.g., $\mathbf A$, $\setg$, $\mathbf a$, $\mathbf g$, etc.), 
and blackboard bold for subsets of $\GG_s$ (e.g., $\setgs$).

\subsection{Reduction of Theorem~\ref{thm:isoperimetric} to isoperimetry in $\mathbf G$}

In this section we show how the initial isoperimetric problem for large subsets of $\mathcal C_R$ can be reduced to 
an isoperimetric problem for large subsets of $\mathbf G$. 
We first show that the set $\mathbf G$ can be viewed as a coarsening of the largest connected subset $\mathcal C_{2R}$ of $\set\cap\ballZ(0,2R)$, 
in particular, that $\mathcal C_{2R}$ and $\mathcal C_R$ are uniquely defined for any configuration in $\mathcal H$ under a mild tuning of the renormalization scales, 
see \eqref{eq:L_0:condition}. 
The key reduction step is formalized in Lemmas~\ref{l:isopmain} and \ref{l:isopbaldG}. 
We finish this section with the proof of Theorem~\ref{thm:isoperimetric} given the results of the lemmas, and prove the lemmas in later sections.

\medskip

We will first show that the largest (in volume) connected component of 
$\set\cap\ballZ(0,2R)$ is uniquely defined, 
and that the set $\mathbf G$ can be viewed as its coarsening on the scale $L_0$.

Recall that $\mathbf G\subset\GG_0\cap\cup_{z_s\in\mathcal G_s}(z_s + [0,L_s)^d)\subset\GG_0\cap\ballZ(0,2R-L_s)$. 
By the definition of a $0$-good vertex and Lemma~\ref{l:fromG0toZd}, 
each of the boxes $x + [0,L_0)^d$, $x\in\mathbf G$, contains a unique connected subset $\mathcal C_x$ of $\set$ of size $\geq \frac 34 \eta(u)L_0^d$, 
and all these sets are connected in $(\mathbf G + [0,2L_0)^d)\subset\ballZ(0,2R)$. 
Therefore, all the $\mathcal C_x$, $x\in\mathbf G$, are part of the same connected component of $\set\cap\ballZ(0,2R)$, which has size 
\[
\geq \left|\bigcup_{x\in\mathbf G}\mathcal C_x\right| \geq \frac 34\eta(u)L_0^d\cdot|\mathbf G| .\
\]
On the other hand, by the definition of a $0$-good vertex and Lemma~\ref{l:fromG0toZd}, 
each of the boxes $x + [0,L_0)^d$, $x\in\mathbf G$, contains $\leq \frac 54 \eta(u)L_0^d$ vertices from $\set_{L_0}$. 
Since in addition by Proposition~\ref{prop:Gproperties}(b), 
\[
|\mathbf G| \geq \left|\GG_s\cap\ballZ(0,2R-2L_s)\right|\cdot f_d\left(\frac{r_0}{l_0}\right)\cdot \left(\frac{L_s}{L_0}\right)^d
\geq
\frac{|\ballZ(0,2R-2L_s)|}{L_s^d}\cdot f_d\left(\frac{r_0}{l_0}\right)\cdot \left(\frac{L_s}{L_0}\right)^d ,\
\]
it follows that
\begin{eqnarray*}
|\set_{L_0}\cap\ballZ(0,2R)| 
&\leq &\frac 54\eta(u)L_0^d \cdot |\mathbf G| + \left(|\ballZ(0,2R)| - L_0^d\cdot|\mathbf G|\right)\\
&\leq 
&L_0^d\cdot |\mathbf G|\cdot \left(
\frac 54\eta(u) + \frac{|\ballZ(0,2R)|}{|\ballZ(0,2R-2L_s)|}\cdot f_d\left(\frac{r_0}{l_0}\right)^{-1} - 1
\right) \\
&\leq 
&L_0^d\cdot |\mathbf G|\cdot \left(
\frac 54\eta(u) + \left(1 - \frac{L_s}{R}\right)^{-d}\cdot f_d\left(\frac{r_0}{l_0}\right)^{-1} - 1
\right) .\
\end{eqnarray*}
Since we assume that $\ei<1$, it follows from \eqref{def:s} that $\frac{L_s}{R} < \frac{1}{L_s}\leq \frac{1}{L_0}$. 
Therefore, there exist $C = C(u)<\infty$ and $\rho = \rho(u)>0$ such that for all 
$L_0>C(u)$ and for all choices of the ratio $\frac{r_0}{l_0}<\rho(u)$, 
\begin{equation}\label{eq:L_0:condition}
\left(1 - \frac{1}{L_0}\right)^{-d}\cdot f_d\left(\frac{r_0}{l_0}\right)^{-1} - 1 < \frac 14\eta(u) .\
\end{equation}
With such a choice of $L_0$, $r_0$, and $l_0$, 
the largest (in volume) connected component of $\set\cap\ballZ(0,2R)$ is uniquely defined and 
\[
\text{$\mathcal C_{2R}$ contains $\mathcal C_x$, for all $x\in\mathbf G$.}
\]
Similar reasoning together with the above conclusion imply that $\mathcal C_{2R}$ contains $\mathcal C_R$.

\bigskip

For any subset $A$ of $\set$, 
we denote by $\partial_\set A$ the edge boundary of $A$ in $\set$, i.e., the set of edges from $\Z^d$ with one end-vertex in $A$ and the other in $\set\setminus A$. 
Similarly, for any subset $\mathbf A$ of $\mathbf G$, 
we denote by $\partial_{\mathbf G} \mathbf A$, the boundary of $\mathbf A$ in $\mathbf G$, i.e., those pairs of vertices in $\GG_0$ which are at $\ell^1$-distance $L_0$ (in $\Z^d$) from each other, 
one of them is in $\mathbf A$ and the other in $\mathbf G\setminus \mathbf A$.

The next two lemmas allow to reduce the initial isoperimetric problem for subsets of $\mathcal C_R$ to an isoperimetric problem for subsets of $\mathbf G$. 
Recall the definition of $\mathcal C_x$ from Lemma~\ref{l:fromG0toZd}(a).

\begin{lemma}\label{l:isopmain}
Let $A$ be a subset of $\mathcal C_R$. 
Let $\setg$ be the set of all $x\in\mathbf G$ such that $\mathcal C_x\cap A \neq\emptyset$, and denote by 
$\setd$ the set of $x\in A$ such that there exists $y\in \mathcal C_{2R}\setminus A$ with $|x-y|_\infty\leq 2L_s$. 
Then
\begin{equation}\label{eq:isopmain2}
|\partial_\set A| 
\geq \max\left\{\frac{1}{d\cdot 2^d}\cdot |\partial_{\mathbf G}\setg|,\frac{|\setd|}{(11\cdot L_s)^d} \right\} ,\
\end{equation}
and there exists $\rho_{\scriptscriptstyle \ref{l:isopmain}}>0$ such that if $\frac{r_0}{l_0}<\rho_{\scriptscriptstyle \ref{l:isopmain}}$ then 
\begin{equation}\label{eq:isopmain1}
|A| \leq 6^d\cdot L_0^d\cdot|\setg| + |\setd| .\
\end{equation}
\end{lemma}
\begin{lemma}\label{l:isopbaldG}
There exist $\gamma_{\scriptscriptstyle \ref{l:isopbaldG}} >0$ and 
$\rho_{\scriptscriptstyle \ref{l:isopbaldG}}>0$ such that if $\frac{r_0}{l_0}<\rho_{\scriptscriptstyle \ref{l:isopbaldG}}$, then 
for any $\mathbf A\subset\mathbf G\cap\ballZ(0,2R-4L_s)$ 
such that $|\mathbf A|\geq 7^{-d}\cdot (\frac{L_s}{L_0})^{2d^2}$, 
we have 
$|\partial_{\mathbf G} \mathbf A|\geq \gamma_{\scriptscriptstyle \ref{l:isopbaldG}}\cdot |\mathbf A|^{\frac{d-1}{d}}$. 
\end{lemma}
Now we finish the proof of Theorem~\ref{thm:isoperimetric} using the two lemmas. 
We prove Lemma~\ref{l:isopmain} in Section~\ref{sec:isopmain} and Lemma~\ref{l:isopbaldG} in Section~\ref{sec:isopbaldG}.

\begin{proof}[Proof of Theorem~\ref{thm:isoperimetric}]
We take $L_0$, $l_0$, and $r_0$ as in \eqref{def:scales} satisfying the statements of Lemmas~\ref{l:Guxk}, \ref{l:isopmain}, and \ref{l:isopbaldG}, and also
\eqref{eq:L_0:condition}. 
We also assume that $4r_0<l_0$ and $5L_s<R$. 
It suffices to show that the event $\mathcal H$ implies the event in \eqref{eq:isoperimetric}.

Fix a subset $A\subset\mathcal C_R$ such that $|A|\geq R^{\ei}$, and define $\setg$ and $\setd$ as in the statement of Lemma~\ref{l:isopmain}. 
Note that $\setg\subset\mathbf G\cap\ballZ(0,R+L_0)\subset\mathbf G\cap\ballZ(0,2R-4L_s)$ if $5L_s<R$. 
First, we claim that 
\begin{equation}\label{eq:ineqmax}
\max\left\{|\partial_{\mathbf G}\setg|,\frac{|\setd|}{L_s^d} \right\}
\geq \min(1,\gamma_{\scriptscriptstyle \ref{l:isopbaldG}})\cdot \max\left\{|\setg|^{\frac{d-1}{d}},\frac{|\setd|}{L_s^d} \right\} .\ 
\end{equation} 
Indeed, if $|\setg|^{\frac{d-1}{d}} < \frac{|\setd|}{L_s^d}$, then \eqref{eq:ineqmax} trivially holds. 
On the other hand, if $|\setg|^{\frac{d-1}{d}} \geq \frac{|\setd|}{L_s^d}$, then using \eqref{eq:isopmain1} we get
\[
|A| \leq 6^d\cdot L_0^d\cdot|\setg| + L_s^d\cdot |\setg|^{\frac{d-1}{d}} \leq 7^d\cdot L_s^d\cdot|\setg| .\
\]
By the assumption on $|A|$ and using \eqref{def:s}, we have that $|A|\geq R^{\ei}\geq L_s^{3d^2}$. 
Hence $|\setg|\geq 7^{-d}\cdot L_s^{2d^2}\geq 7^{-d}\cdot (\frac{L_s}{L_0})^{2d^2}$, 
and \eqref{eq:ineqmax} follows from Lemma~\ref{l:isopbaldG} applied to $\mathbf A = \setg$.

\bigskip

It follows from \eqref{eq:isopmain2}, \eqref{eq:isopmain1}, and \eqref{eq:ineqmax} that
\begin{multline*}
\frac{|\partial_\set A|}{|A|^{\frac{d-1}{d}}}
\geq  \frac{\frac{1}{d\cdot 11^d}\cdot\min(1,\gamma_{\scriptscriptstyle \ref{l:isopbaldG}})\cdot\max\left\{|\setg|^{\frac{d-1}{d}},\frac{|\setd|}{L_s^d} \right\}}
{(6^d\cdot L_0^d\cdot|\setg| + |\setd|)^{\frac{d-1}{d}}}\\
\geq  \frac{\frac{1}{d\cdot 11^d}\cdot\min(1,\gamma_{\scriptscriptstyle \ref{l:isopbaldG}})\cdot \max\left\{|\setg|^{\frac{d-1}{d}},\frac{|\setd|}{L_s^d} \right\}}
{6^{d-1}\cdot L_0^{d-1}\cdot|\setg|^{\frac{d-1}{d}} + |\setd|^{\frac{d-1}{d}}} .\
\end{multline*}
On the one hand, if $L_0^d\cdot |\setg| \geq |\setd|$, then 
\[
\frac{|\partial_\set A|}{|A|^{\frac{d-1}{d}}}
\geq  \frac{\frac{1}{d\cdot 11^d}\cdot\min(1,\gamma_{\scriptscriptstyle \ref{l:isopbaldG}})\cdot |\setg|^{\frac{d-1}{d}}}
{6^{d-1}\cdot L_0^{d-1}\cdot|\setg|^{\frac{d-1}{d}} + |\setd|^{\frac{d-1}{d}}} 
\geq \frac{\frac{1}{d\cdot 11^d}\cdot\min(1,\gamma_{\scriptscriptstyle \ref{l:isopbaldG}})}{(7\cdot L_0)^{d-1}} .\
\]
On the other hand, if $L_0^d\cdot |\setg| \leq |\setd|$, then by \eqref{eq:isopmain1} and \eqref{def:s}, 
$|\setd| \geq 7^{-d}\cdot |A| \geq 7^{-d}\cdot R^{\ei}\geq 7^{-d}\cdot L_s^{3d^2}$, and we get
\[
\frac{|\partial_\set A|}{|A|^{\frac{d-1}{d}}}
\geq  \frac{\frac{1}{d\cdot 11^d}\cdot\min(1,\gamma_{\scriptscriptstyle \ref{l:isopbaldG}})\cdot|\setd|^{\frac 1d}}{7^{d-1}\cdot L_s^d} 
\geq \frac{1}{d\cdot 77^d}\cdot \min(1,\gamma_{\scriptscriptstyle \ref{l:isopbaldG}}) .\
\]
This finishes the proof of Theorem~\ref{thm:isoperimetric} with the choice of 
$\gamma_{\scriptscriptstyle \ref{thm:isoperimetric}} = \frac{1}{d\cdot 77^d\cdot L_0^{d-1}}\cdot\min(1,\gamma_{\scriptscriptstyle \ref{l:isopbaldG}})$.
\end{proof}

\subsection{Proof of Lemma~\ref{l:isopmain}}\label{sec:isopmain}
\noindent

The proof of both \eqref{eq:isopmain2} and \eqref{eq:isopmain1} goes by constructing certain mappings 
from $\partial_{\mathbf G}\setg$ to $\partial_\set A$ (a $d\cdot 2^d$ to $1$ map), from $\setd$ to $\partial_\set A$ (a $(11\cdot L_s)^d$ to $1$ map), 
and from $A\setminus\setd$ to $\setg$ (a $(6\cdot L_0)^d$ to $1$ map). 

\medskip

Recall the definition of $\mathcal C_x$ from Lemma~\ref{l:fromG0toZd}(a). 

\medskip

\begin{proof}[Proof of \eqref{eq:isopmain2}] 
Note that for any $x\in\setg$ and $y\in \mathbf G\setminus \setg$ such that $|x-y|_1 = L_0$, 
$\mathcal C_x\cap A\neq \emptyset$ and $\mathcal C_y \subset \mathcal C_{2R} \setminus A$. 
By Lemma~\ref{l:fromG0toZd}, $\mathcal C_x$ and $\mathcal C_y$ are connected in $\set\cap((x+[0,2L_0)^d)\cup(y+[0,2L_0)^d))$. 
Each path in $\set$ connecting $\mathcal C_x\cap A$ and $\mathcal C_y$ contains an edge from $\partial_{\set} A$. 
This implies that 
\[
|\partial_\set A| \geq \frac{1}{d\cdot 2^d}\cdot |\partial_{\mathbf G}\setg| ,\
\]
where the constant $\frac{1}{d\cdot 2^d}$ takes care for overcounting. 
We next show that 
\[
|\partial_\set A| \geq \frac{|\setd|}{(11\cdot L_s)^d} .\
\]
Indeed, by the definition of $\setd$, for any $x\in \setd$, there exists $y\in \mathcal C_{2R}\setminus A$ such that $|x-y|_\infty\leq 2L_s$. 
By the second part of the definition of $\mathcal H$, we conclude that any such $x$ and $y$ are connected by a path in $\set\cap\ballZ(x,4L_s)$.  
Since $x\in A$ and $y\notin A$, this path necessarily contains an edge from $\partial_\set A$. 
This implies that 
\[
|\setd| \leq |\{(x,e)~:~x\in \setd, e\in \partial_\set A\cap\ballZ(x,4L_s)\}|\leq |\ballZ(0,4L_s+1)|\cdot |\partial_\set A| ,\
\]
and the claim follows.
\end{proof}

\bigskip

\noindent
\begin{proof}[Proof of \eqref{eq:isopmain1}]
We need to show that $|A\setminus\setd| \leq 6^d\cdot L_0^d\cdot |\setg|$. 

We choose $\rho_{\scriptscriptstyle \ref{l:isopmain}}>0$ such that if $\frac{r_0}{l_0}<\rho_{\scriptscriptstyle \ref{l:isopmain}}$ then 
\begin{equation}\label{eq:r_0l_0:isopmain}
f_d\left(\frac{r_0}{l_0}\right)>\frac 12 .\
\end{equation}
Then, by Proposition~\ref{prop:Gproperties}(b), for any $z\in \ballZ(0,R)$, the box $(z+[-2L_s,2L_s)^d)$ contains at least $(\frac{L_s}{L_0})^d$ vertices from $\mathbf G$.
Since for any $y\in\mathbf G$, $\mathcal C_y\subset \mathcal C_{2R}$, 
we have that for any $x\in A\setminus \setd$, the set $x + [-2L_s,2L_s)^d$ contains at least $(\frac{L_s}{L_0})^d$ vertices from $\setg$. 
(Mind that every vertex from $\mathcal C_{2R}\cap(x+[-2L_s,2L_s)^d$ must be in $A$ by the definition of $\setd$.)
Thus we have a map from $A\setminus\setd$ to a subset of $\setg$ of size at least $(\frac{L_s}{L_0})^d$ such that every vertex of $\setg$ is 
in the image of at most $6^d\cdot L_s^d$ vertices from $A\setminus\setd$. 
This implies that $|A\setminus\setd| \leq 6^d\cdot L_0^d\cdot |\setg|$, and \eqref{eq:isopmain1} is proved. 
\end{proof}

\subsection{Proof of Lemma~\ref{l:isopbaldG}}\label{sec:isopbaldG}

The statement of the lemma concerns with sets $\mathbf A\subset\mathbf G\cap\ballZ(0,2R-4L_s)$ of large enough size, 
but not necessarily comparable with the size of $\GG_0\cap\ballZ(0,2R-4L_s)$.  
We distinguish the cases when $\mathbf A$ is sparse, and when it is localized. 
In the first case, we prove that the boundary of $\mathbf A$ is almost of the same size as the volume of $\mathbf A$. 
In the second case, we estimate the boundary of $\mathbf A$ locally in each of the boxes of side length $3L_s$ which 
has dense intersection with $\mathbf A$ and with its complement, see Lemma~\ref{l:isopa}. 
More precisely, we show that the boundary of $\mathbf A$ in each of these boxes is at least of order $L_s^{d-1}$. 
Using the isoperimetric inequality for subsets of the lattice $\GG_s$ we show that the number of disjoint such boxes is of order $\frac{|\mathbf A|^{\frac{d-1}{d}}}{L_s^{d-1}}$. 
Thus we show that the boundary of $\mathbf A$ contains an order of $\frac{|\mathbf A|^{\frac{d-1}{d}}}{L_s^{d-1}}$ disjoint pieces 
of size $L_s^{d-1}$. 
Before we proceed with the proof, we state the key ingredient of the proof as Lemma~\ref{l:isopa}.

\begin{lemma}\label{l:isopa}
Let $x\in\GG_s\cap\ballZ(0,2R-3L_s)$. Denote by $\mathbf g$ the graph $\mathbf G\cap(x+[-L_s,2L_s)^d)$. 
For any subset $\mathbf a$ of $\mathbf g$, let $\partial_{\mathbf g} \mathbf a$ be the set of pairs of vertices in $\mathbf g$ 
at $\ell^1$-distance $L_0$ (in $\Z^d$) from each other so that one of them is in $\mathbf a$ and the other in $\mathbf g\setminus \mathbf a$.

There exists $\gamma_{\scriptscriptstyle \ref{l:isopa}}>0$ and $\rho_{\scriptscriptstyle \ref{l:isopa}}>0$ such that 
if $\frac{r_0}{l_0}<\rho_{\scriptscriptstyle \ref{l:isopa}}$ then 
for any subset $\mathbf a$ of $\mathbf g$ with $|\mathbf a|\in[\frac 12 (\frac{L_s}{L_0})^d, (3^d-\frac 12)(\frac{L_s}{L_0})^d]$, 
we have $|\partial_{\mathbf g} \mathbf a| \geq \gamma_{\scriptscriptstyle \ref{l:isopa}}\cdot (\frac{L_s}{L_0})^{d-1}$.
\end{lemma}
We postpone the proof of Lemma~\ref{l:isopa} until Section~\ref{sec:isopa}, and now show how Lemma~\ref{l:isopbaldG} follows from Lemma~\ref{l:isopa}. 

Take $\mathbf A\subset\mathbf G\cap\ballZ(0,2R-4L_s)$ such that $|\mathbf A| \geq 7^{-d}\cdot(\frac{L_s}{L_0})^{2d^2}$. 
Note that 
\[
\left|\{x\in\GG_s~:~\mathbf A\cap(x+[0,L_s)^d)\neq\emptyset\}\right| \geq |\mathbf A|\cdot\left(\frac{L_s}{L_0}\right)^{-d} .\
\]
Let $\setgs$ be the set of $x\in\GG_s$ such that 
\begin{equation}\label{eq:setgs}
|\mathbf A\cap(x+[0,L_s)^d)| \geq \frac 12\cdot \left(\frac{L_s}{L_0}\right)^d .\
\end{equation}
Note that $\setgs\subset\GG_s\cap\ballZ(0,2R-3L_s)$. 

By Proposition~\ref{prop:Gproperties}(b), for any $x\in\GG_s\cap\ballZ(0,2R-2L_s)$, $|\mathbf G\cap(x+[0,L_s)^d)|>f_d(\frac{r_0}{l_0})\cdot(\frac{L_s}{L_0})^d$, 
where $f_d$ is defined in \eqref{eq:fj}. 
By \eqref{eq:r_0l_0:isopmain}, for any choice of the ratio $\frac{r_0}{l_0}<\rho_{\scriptscriptstyle \ref{l:isopmain}}$, $f_d(\frac{r_0}{l_0})>\frac{1}{2}$. 
This implies that for any $x\in\GG_s\cap\ballZ(0,2R-2L_s)$, $|\mathbf G\cap(x+[0,L_s)^d)| > \frac 12\cdot(\frac{L_s}{L_0})^d$. 
Thus, if $|\setgs|< \frac 12\cdot |\mathbf A|\cdot (\frac{L_s}{L_0})^{-d}$, then 
the number of $x\in\GG_s$ such that $x + [0,L_s)^d$ intersects {\it both} $\mathbf A$ and $\mathbf G\setminus \mathbf A$
is at least $\frac 12 \cdot |\mathbf A|\cdot (\frac{L_s}{L_0})^{-d}$. 
Since $\mathbf G\cap(x+[0,L_s)^d)$ is connected for each such $x$, 
there is an edge in $\partial_{\mathbf G} \mathbf A$ with both end-vertices in $\mathbf G\cap(x+[0,L_s)^d)$. Therefore,
\[
|\partial_{\mathbf G} \mathbf A|\geq 
\frac 12 \cdot |\mathbf A|\cdot \left(\frac{L_s}{L_0}\right)^{-d}
\geq \frac{1}{14}\cdot |\mathbf A|^{\frac{d-1}{d}}\cdot \left(\frac{L_s}{L_0}\right)^{2d}\cdot \left(\frac{L_s}{L_0}\right)^{-d} 
\geq \frac{1}{14}\cdot |\mathbf A|^{\frac{d-1}{d}} .\
\]
Assume now that $|\setgs|\geq \frac 12\cdot |\mathbf A|\cdot (\frac{L_s}{L_0})^{-d}$. 
Let $\partial_{\GG_s} \setgs$ be the set of edges of $\GG_s$ with exactly one end-vertex in $\setgs$. 
By the isoperimetric inequality on $\GG_s$, 
\[
|\partial_{\GG_s} \setgs|\geq c_\star \cdot |\setgs|^{\frac{d-1}{d}}
\geq \frac 12\cdot c_\star\cdot |\mathbf A|^{\frac{d-1}{d}}\cdot \left(\frac{L_s}{L_0}\right)^{1-d} ,\
\]
where $c_\star>0$ is the isoperimetric constant for $\Z^d$.
By the definition of $\setgs$, for any $y\in\GG_s\setminus\setgs$, $|\mathbf A\cap(y+[0,L_s)^d)| \leq \frac 12 \cdot (\frac{L_s}{L_0})^d$. 
Therefore, for any $x\in \setgs$ such that there exists $y\in\GG_s\setminus\setgs$ with $|x-y|_1 = L_s$, 
$\frac 12 \cdot (\frac{L_s}{L_0})^d\leq |\mathbf A\cap(x+[-L_s,2L_s)^d)| \leq (3^d - \frac 12) \cdot (\frac{L_s}{L_0})^d$.
Since $\setgs\subset\GG_s\cap\ballZ(0,2R-3L_s)$, 
we can apply Lemma~\ref{l:isopa} to $\mathbf g = \mathbf G\cap(x+[-L_s,2L_s)^d)$ and $\mathbf a = \mathbf A\cap(x+[-L_s,2L_s)^d)$, to obtain that 
if $\frac{r_0}{l_0}<\rho_{\scriptscriptstyle \ref{l:isopa}}$, then 
\[
|\partial_{\mathbf G}\mathbf A \cap(x+[-L_s,2L_s)^d)|\geq \gamma_{\scriptscriptstyle \ref{l:isopa}}\cdot \left(\frac{L_s}{L_0}\right)^{d-1} .\
\] 
We are essentially done. 
Let ${\sum}^*$ be the sum over $x\in\setgs$ such that there exists $y\in\GG_s\setminus\setgs$ with $|x-y|_1 = L_s$, 
i.e., $\{x,y\}\in \partial_{\GG_s} \setgs$. 
Combining the last two estimates we get
\begin{multline*}
|\partial_{\mathbf G} \mathbf A| \geq 
\frac{1}{3^d}\cdot {\sum}^*
|\partial_{\mathbf G}\mathbf A \cap(x+[-L_s,2L_s)^d)|\\
\geq \frac{1}{3^d}\cdot\left(\frac{1}{2d}\cdot \frac 12\cdot c_\star\cdot |\mathbf A|^{\frac{d-1}{d}}\cdot \left(\frac{L_s}{L_0}\right)^{1-d}\right)
\cdot\left(\gamma_{\scriptscriptstyle \ref{l:isopa}}\cdot \left(\frac{L_s}{L_0}\right)^{d-1}\right)
\geq \left(\frac {1}{4d\cdot 3^d}\cdot c_\star\cdot\gamma_{\scriptscriptstyle \ref{l:isopa}}\right)\cdot |\mathbf A|^{\frac{d-1}{d}} .\
\end{multline*}
Our final choice of $\rho_{\scriptscriptstyle \ref{l:isopbaldG}}$ and $\gamma_{\scriptscriptstyle \ref{l:isopbaldG}}$ is 
\[
\rho_{\scriptscriptstyle \ref{l:isopbaldG}} = \min\left(\rho_{\scriptscriptstyle \ref{l:isopmain}},\rho_{\scriptscriptstyle \ref{l:isopa}}\right) \quad\mbox{and}\quad
\gamma_{\scriptscriptstyle \ref{l:isopbaldG}} = \min\left(\frac{1}{14},\frac {1}{4d\cdot 3^d}\cdot c_\star\cdot\gamma_{\scriptscriptstyle \ref{l:isopa}}\right) ,\
\]
where $c_\star$ is the isoperimetric constant for $\Z^d$.
The proof of Lemma~\ref{l:isopbaldG} is complete, subject to Lemma~\ref{l:isopa}.\qed

\subsection{Proof of Lemma~\ref{l:isopa}}\label{sec:isopa}

We would like to prove that for any subset $\mathbf a$ of $\mathbf g$ which occupies a non-trivial (bounded away from $0$ and $1$) fraction of vertices in 
$\GG_0\cap(x_s + [-L_s,2L_s)^d)$, $x_s\in\GG_s$, 
its boundary is at least an order of $(\frac{L_s}{L_0})^{d-1}$. 
For this we prove a much stronger statement that for any $j$-dimensional ($2\leq j\leq d$) subbox of $\GG_0\cap(x_s + [-L_s,2L_s)^d)$ containing a non-trivial 
fraction of vertices of $\mathbf a$, the boundary of $\mathbf a$ in the restriction of $\mathbf g$ to this $j$-dimensional subbox is at least an order of $(\frac{L_s}{L_0})^{j-1}$, 
see Lemma~\ref{l:isopaj}. 
This statement is proved by induction on $j$. 
The case $j=2$ is the most involved. 
We first reduce the problem to connected sets with complement consisting of large connected components, see \eqref{eq:isopaconnected}. 
The boundary (in $\GG_0$) of such sets is large (see \eqref{eq:isopbox}) and consists of only large $*$-connected pieces (see \eqref{eq:lbDelta}). 
The key step in the proof is to show that each individual $*$-connected piece of the boundary consists mostly of the edges from $\mathbf g$ 
(see \eqref{eq:relationboundaries}). This is done by exploiting further the multi-scale construction of $\mathbf G$. 
In the case $j\geq 3$, we use a dimension reduction argument. 
We partition the $j$-dimensional box into smaller dimensional subboxes, and estimate the part of the boundary of $\mathbf a$ in each 
individual subbox where $\mathbf a$ has a non-trivial density. 

\bigskip

The main result of this section is the following lemma.
\begin{lemma}\label{l:isopaj}
For any $x\in\GG_s\cap\ballZ(0,2R-3L_s)$, $y\in\GG_0\cap(x+[-L_s,2L_s)^d)$, $2\leq j\leq d$, and pairwise orthogonal $e_1,\dots,e_j\in\Z^d$ with $|e_i|_1=1$, 
let $\mathbf g'$ be the restriction of $\mathbf g$ to the $j$-dimensional subcube $(y + \sum_{i=1}^j\Z\cdot e_i)\cap(x+[-L_s,2L_s)^d)$ of $(x+[-L_s,2L_s)^d)$. 
For any $\epsilon>0$ there exists $\gamma_{\scriptscriptstyle \ref{l:isopaj}} = \gamma_{\scriptscriptstyle \ref{l:isopaj}}(\epsilon, j)>0$ and 
$\rho_{\scriptscriptstyle \ref{l:isopaj}} = \rho_{\scriptscriptstyle \ref{l:isopaj}}(\epsilon,j)>0$ such that 
if $\frac{r_0}{l_0}<\rho_{\scriptscriptstyle \ref{l:isopaj}}$ then 
for any subset $\mathbf a'$ of $\mathbf g'$ with $|\mathbf a'|\in[\epsilon (\frac{3L_s}{L_0})^j, (1-\epsilon)(\frac{3L_s}{L_0})^j]$,
we have $|\partial_{\mathbf g'} \mathbf a'| \geq \gamma_{\scriptscriptstyle \ref{l:isopaj}}(\epsilon,j) (\frac{L_s}{L_0})^{j-1}$.
\end{lemma}
Note that Lemma~\ref{l:isopa} is a special case of Lemma~\ref{l:isopaj} corresponding to the choice of $j=d$ and $\epsilon = \frac{1}{2\cdot 3^{d}}$. 
In particular, Lemma~\ref{l:isopaj} implies Lemma~\ref{l:isopa} with the choice of 
$\rho_{\scriptscriptstyle \ref{l:isopa}} = \rho_{\scriptscriptstyle \ref{l:isopaj}}(\frac{1}{2\cdot 3^d},d)$ and 
$\gamma_{\scriptscriptstyle \ref{l:isopa}} = \gamma_{\scriptscriptstyle \ref{l:isopaj}}(\frac{1}{2\cdot 3^d},d)$. 
Thus it only remains to prove Lemma~\ref{l:isopaj}.
We first prove Lemma~\ref{l:isopaj} in the case $j=2$, and then use induction on $j$ to prove Lemma~\ref{l:isopaj} in the case $j\geq 3$.  

\medskip

\begin{proof}[Proof of Lemma~\ref{l:isopaj} ($j=2$)]
Fix any pair of orthogonal $e_1,e_2\in\Z^d$ with $|e_i|_1 = 1$, $y\in\GG_0\cap(x+[-L_s,2L_s)^d)$, 
and let 
\[
Q = \GG_0\cap(x+[-L_s,2L_s)^d)\cap(y + \Z\cdot e_1 + \Z\cdot e_2) .\
\] 
Denote by $\mathbf g'$ the restriction of $\mathbf g$ to $Q$. 
By Proposition~\ref{prop:Gproperties}(c), $\mathbf g'$ is connected and 
$|\mathbf g'|> f_2(\frac{r_0}{l_0})\cdot(\frac{3L_s}{L_0})^2$. 
We choose $\rho_1 = \rho_1(\epsilon)>0$ so that if $\frac{r_0}{l_0}<\rho_1$ then 
\begin{equation}\label{eq:fj:isopaj}
f_2\left(\frac{r_0}{l_0}\right)> 1 - \frac{\epsilon}{4},
\end{equation}
which implies that $|\mathbf g'| \geq (1 - \frac{\epsilon}{4})\cdot(\frac{3L_s}{L_0})^2$. 

Next, we claim that it suffices to show that there exists $c=c(\epsilon)>0$ and $\rho_2>0$ such that if $\frac{r_0}{l_0}<\rho_2$, then   
\begin{equation}\label{eq:isopaconnected}
\begin{array}{c}
\text{for any {\it connected} subset $\mathbf a''$ of $\mathbf g'$ satisfying $|\mathbf a''|\in[(\frac{3L_s}{L_0}), (1-\frac{3\epsilon}{8})(\frac{3L_s}{L_0})^2]$}\\
\text{and such that each connected component of $\mathbf g'\setminus \mathbf a''$ has size $\geq (\frac{3L_s}{L_0})$,}\\
\text{we have $|\partial_{\mathbf g'} \mathbf a''| \geq c(\epsilon)\cdot |\mathbf a''|^{1/2}$.}
\end{array}
\end{equation}
Indeed, assume that $\mathbf a' = \mathbf a_1'\cup\mathbf a_2'$, where $\mathbf a_1'$ is the subset of $\mathbf a'$ 
consisting of connected components of $\mathbf a'$ of size $\geq \frac{3L_s}{L_0}$, and $\mathbf a_2'$ is the rest of $\mathbf a'$. 
Let $N$ be the number of connected components in $\mathbf a_2'$. 
If $|\mathbf a_1'| \leq |\mathbf a_2'| (\leq \frac{3L_s}{L_0}\cdot N)$, then $N\geq \frac 12\cdot \frac{L_0}{3L_s}\cdot |\mathbf a'|\geq \frac{\epsilon}{2}\cdot (\frac{3L_s}{L_0})$, and 
\[
\frac{|\partial_{\mathbf g'} \mathbf a'|}{|\mathbf a'|^{1/2}}
\geq 
\frac{N}{(\frac{6L_s}{L_0}\cdot N)^{1/2}} \geq \frac 12\cdot \epsilon^{1/2} .\
\]
On the other hand, if $|\mathbf a_1'| > |\mathbf a_2'|$, then
\[
\frac{|\partial_{\mathbf g'} \mathbf a'|}{|\mathbf a'|^{1/2}} 
\geq 
\frac{1}{\sqrt{2}}\cdot \frac{|\partial_{\mathbf g'} \mathbf a_1'|}{|\mathbf a_1'|^{1/2}} .\
\]
The same reasoning applied to $\mathbf g'\setminus\mathbf a_1'$ implies that 
we may assume that the total volume of connected components of $\mathbf g'\setminus \mathbf a_1'$ with size $< (\frac{3L_s}{L_0})$ 
is at most $\frac 12 |\mathbf g'\setminus\mathbf a_1'|$. 
By merging all these small connected components of $\mathbf g'\setminus \mathbf a_1'$ into $\mathbf a_1'$, 
we obtain the set $\mathbf a'''$ such that $|\partial_{\mathbf g'} \mathbf a'''| \leq |\partial_{\mathbf g'} \mathbf a_1'|$, 
all connected components of $\mathbf a'''$ 
and $\mathbf g'\setminus \mathbf a'''$ have size $\geq \frac{3L_s}{L_0}$, 
and $|\mathbf g'\setminus \mathbf a'''| \geq \frac 12 |\mathbf g'\setminus\mathbf a_1'| \geq \frac 12 |\mathbf g'\setminus\mathbf a'| \geq \frac{3\epsilon}{8}\cdot(\frac{3L_s}{L_0})^2$.
Moreover, using the same ideas as, e.g., in \cite[Section~3.1]{MathieuRemy}, we get that for some $c>0$
\[
\frac{|\partial_{\mathbf g'} \mathbf a_1'|}{|\mathbf a_1'|^{1/2}}
\geq
\frac{|\partial_{\mathbf g'} \mathbf a'''|}{|\mathbf a'''|^{1/2}}
\geq c\cdot \inf_{\mathbf a''}\frac{|\partial_{\mathbf g'} \mathbf a''|}{|\mathbf a''|^{1/2}} ,\
\]
where the infimum is over all {\it connected} subsets $\mathbf a''$ of $\mathbf g'$ with $|\mathbf a''|\in[(\frac{3L_s}{L_0}), (1-\frac{3\epsilon}{8})(\frac{3L_s}{L_0})^2]$
and such that each connected component of $\mathbf g'\setminus \mathbf a''$ has size $\geq \frac{3L_s}{L_0}$.
Thus, if \eqref{eq:isopaconnected} holds, then Lemma~\ref{l:isopaj} follows in the case $j=2$ with the choice of 
\[
\rho_{\scriptscriptstyle \ref{l:isopaj}}(\epsilon,2) = \min(\rho_1(\epsilon),\rho_2) .\
\]

We proceed with the proof of \eqref{eq:isopaconnected}. 
Here we will need the full strength of property (a) in the definition of the event $\mathcal H$ (see Remark~\ref{rem:sgood}). 
We will also use the definition of sets $(\mathcal G_i)_{0\leq i\leq r}$ from the construction of $\mathbf G$ (see Remark~\ref{rem:Gexception}).  

Recall that $r = \lfloor \frac{s}{2} \rfloor$. It follows from \eqref{eq:scalesrs} that  
\begin{equation}\label{eq:sr}
\frac{3L_s}{L_0} \geq 3 \left(\frac{L_r}{L_0}\right)^2 .\
\end{equation}
Let $\mathbf b''$ be the connected components (in $\GG_0$) of $Q\setminus \mathbf a''$ which do not intersect $\mathbf g'\setminus \mathbf a''$, 
and let $\overline {\mathbf a} = \mathbf a''\cup\mathbf b''$.
(In other words, $\overline {\mathbf a}$ is obtained from $\mathbf a''$ by ``filling in holes'' in $\mathbf a''$, see Figure~\ref{fig:bara}.)

\begin{figure}
\centering
\includegraphics[width=0.4\textwidth]{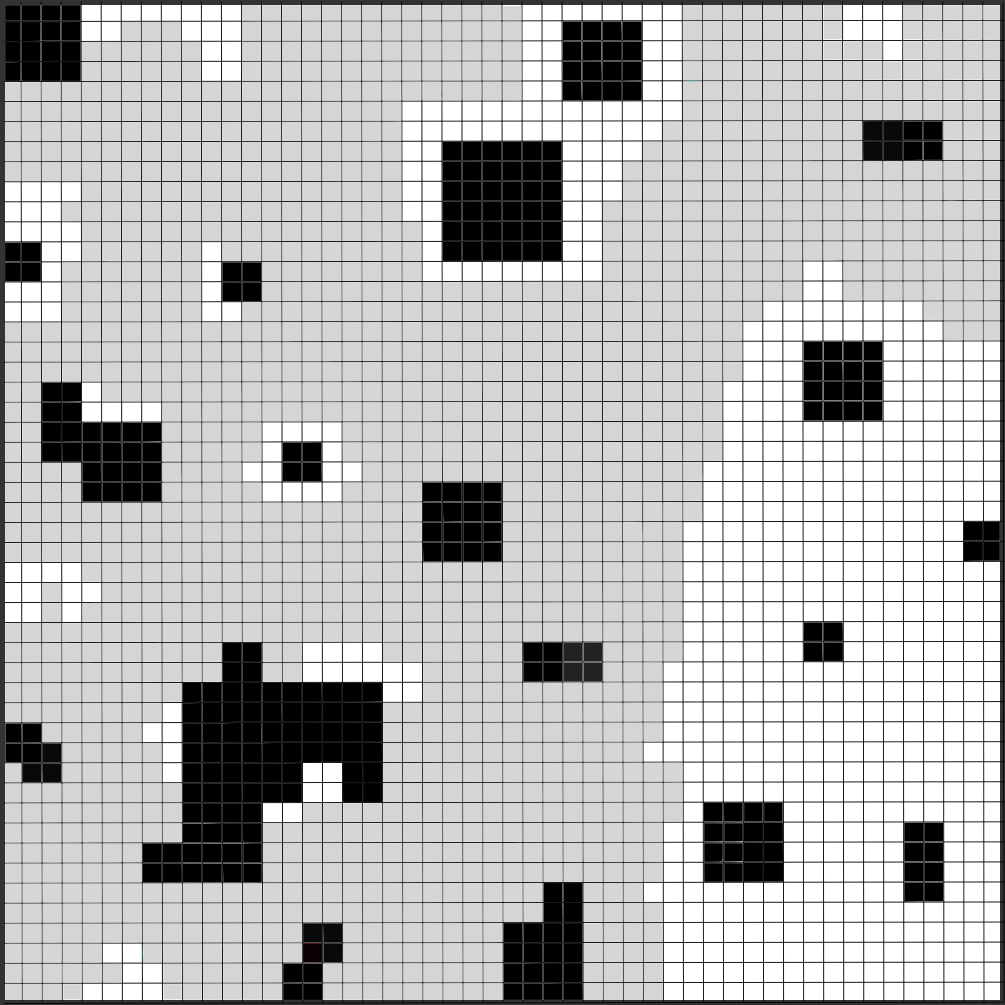}
\hspace{2cm}
\includegraphics[width=0.4\textwidth]{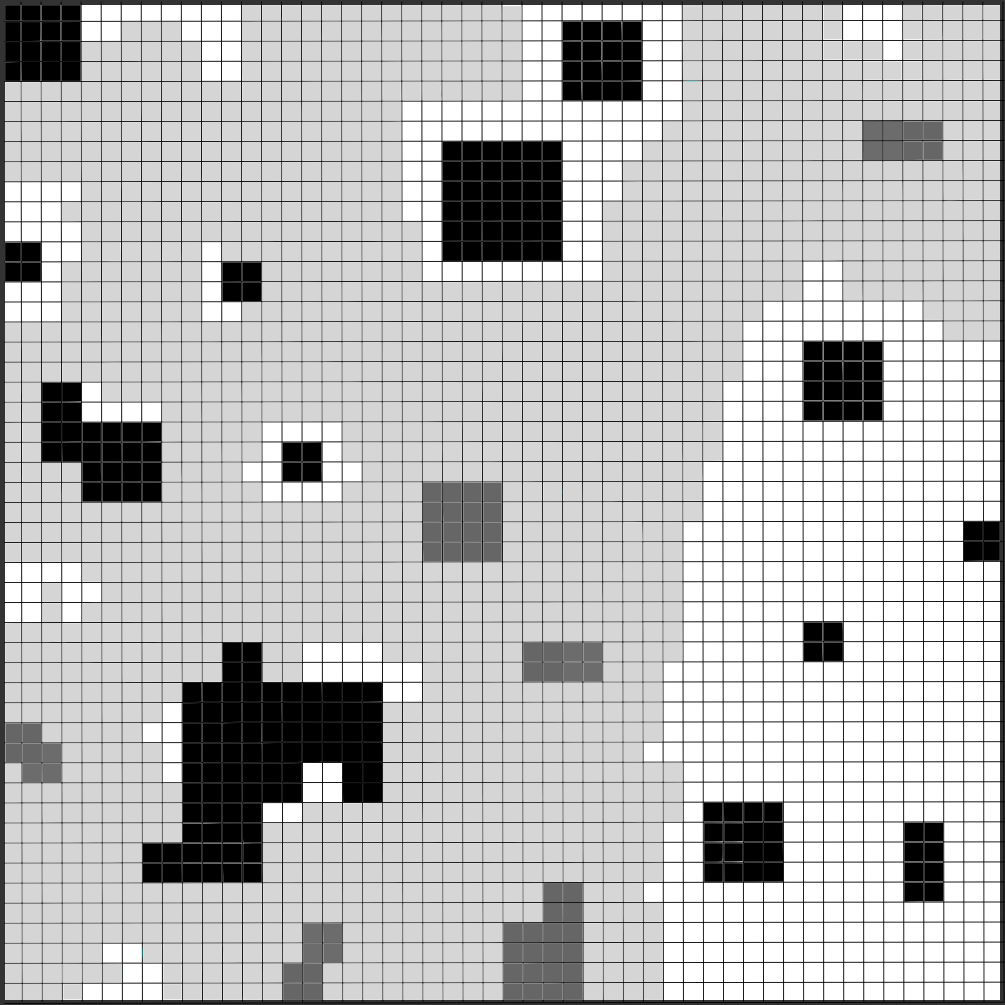}
\caption{This is an illustration of the set $\overline {\mathbf a}$. The black region on the left picture corresponds to $Q\setminus\mathbf g'$, 
but the black boxes are not drawn to the actual scale. The light grey region corresponds to $\mathbf a''$, the white to $\mathbf g'\setminus\mathbf a''$, and 
the grey on the right picture to $\mathbf b''$. Thus, the union of light grey and grey regions corresponds to $\overline {\mathbf a}$. 
Note that a black box turns grey only if it does not have a white neighbor.}\label{fig:bara}
\end{figure}

Note that $\overline {\mathbf a}$ is connected, each connected component of $Q\setminus \overline{\mathbf a}$ has size $\geq \frac{3L_s}{L_0}$, 
$|\overline{\mathbf a}| \geq |\mathbf a''| \geq \frac{3L_s}{L_0}$, and 
$|Q\setminus\overline{\mathbf a}| \geq |Q| - |\mathbf a''| - |Q\setminus \mathbf g'| \geq \frac{\epsilon}{8}\cdot (\frac{3L_s}{L_0})^2$. 

Let $(\overline{\mathbf b}_i)_{i\geq 1}$ be connected components of $Q\setminus\overline{\mathbf a}$. 
Let $\Delta_i$ be the set of edges $\{x,y\}$ with $x\in\overline{\mathbf a}$ and $y\in\overline{\mathbf b}_i$ 
(note that necessarily $x\in\mathbf a''$ by the definition of $\overline{\mathbf a}$ and $\overline{\mathbf b}_i$), 
and denote by $\delta_i$ the set of edges $\{x,y\}\in\Delta_i$ such that $x\in\overline{\mathbf a}$ and $y\in \mathbf g'\setminus\overline{\mathbf a}$. 
Note that 
\[
|\partial_{\mathbf g'}\mathbf a''| = \sum_{i\geq 1}|\delta_i| \qquad\mbox{and}\qquad
|\partial_Q\overline{\mathbf a}| = \sum_{i\geq 1}|\Delta_i| .\
\]
We will show that there exists $C<\infty$ such that for all $i\geq 1$, 
\begin{equation}\label{eq:relationboundaries}
|\delta_i| \geq |\Delta_i|\cdot\left(1 - C\sum_{k=0}^{r-1}\frac{r_k}{l_k}\right) .\ 
\end{equation}
Once \eqref{eq:relationboundaries} is proved, we choose $\rho_2>0$ so that for $\frac{r_0}{l_0}<\rho_2$, 
\begin{equation}\label{eq:r_0l_0:isopaj2}
1 - C\sum_{k=0}^{r-1}\frac{r_k}{l_k}>\frac 12 .\
\end{equation}
Then using the fact that $|\overline {\mathbf a}|\in[(\frac{3L_s}{L_0}),(1-\frac{\epsilon}{8})(\frac{3L_s}{L_0})^2]$, 
we apply the isoperimetric inequality for $\overline{\mathbf a}$ in $Q$ (see \cite[Proposition~2.2]{DeuschelPisztora}) and get
\begin{multline}\label{eq:isopbox}
|\mathbf a''| \leq |\overline{\mathbf a}| \leq C(\epsilon)\cdot\sum_{i\geq 1}|\Delta_i|^2 
\leq C(\epsilon)\cdot \left(\sum_{i\geq 1}|\Delta_i|\right)^2
\leq 4\cdot C(\epsilon)\cdot \left(\sum_{i\geq 1}|\delta_i|\right)^2\\
= 4\cdot C(\epsilon)\cdot |\partial_{\mathbf g'}\mathbf a''|^2 ,\
\end{multline}
and \eqref{eq:isopaconnected} follows. 
Before we prove \eqref{eq:relationboundaries}, we show that there exists $c>0$ such that for each $i\geq 1$, 
\begin{equation}\label{eq:lbDelta}
|\Delta_i| \geq c\cdot \frac{L_r}{L_0} .\
\end{equation}
Indeed, if $|\overline{\mathbf b}_i|<\frac 12(\frac{3L_s}{L_0})^2$, then 
by the isoperimetric inequality in $Q$ (see \cite[Proposition~2.2]{DeuschelPisztora}), 
$|\Delta_i| \geq c|\overline{\mathbf b}_i|^{1/2} \geq c(\frac{3L_s}{L_0})^{1/2}$. 
On the other hand, if $|\overline{\mathbf b}_i|\geq \frac 12(\frac{3L_s}{L_0})^2$, 
then $Q\setminus \overline{\mathbf b}_i$ is connected (since $\overline{\mathbf a}$ is connected), 
and $|Q\setminus \overline{\mathbf b}_i|\in [(\frac{3L_s}{L_0}), \frac 12(\frac{3L_s}{L_0})^2]$. 
Thus, again by the isoperimetric inequality in $Q$ (see \cite[Proposition~2.2]{DeuschelPisztora}), 
$|\Delta_i| \geq c|Q\setminus \overline{\mathbf b}_i|^{1/2} \geq c(\frac{3L_s}{L_0})^{1/2}$.
Using \eqref{eq:sr}, we get \eqref{eq:lbDelta}.

\medskip

We now prove \eqref{eq:relationboundaries}. 
For this we recall the construction of $\mathbf G$, namely the definition of $\mathcal G_k$. 
In particular, note that by part (a) of the definition of $\mathcal H$, $\mathcal G_r \cap \ballZ(0,R+L_r) = \GG_r\cap\ballZ(0,R+L_r)$, 
and for $0\leq k\leq r-1$, $\mathcal G_k$ is obtained by deleting at most $3$ boxes of side length $r_kL_k$ 
from each of the boxes $(z + [0,L_{k+1})^2)$, $z\in \mathcal G_{k+1}$.  
A useful implication of this construction is that 
for each such deleted box of side length $r_kL_k$, 
there exist at most $26 (= 3\cdot 3^2 - 1)$ other deleted boxes of side length $r_kL_k$ 
which are within $\ell^\infty$-distance $L_{k+1}$ from the specified box. 

Fix $i\geq 1$. 
We write the set of ``bad'' edges $\Delta_i\setminus\delta_i$ as the union $\cup_{k=0}^{r-1}E_k$, 
where $E_k$ consists of edges $\{z,z'\}$ in $\GG_0$ such that $z\in\overline{\mathbf a}$ and 
$z'\in \overline{\mathbf b}_i \cap((\mathcal G_{k+1} + [0,L_{k+1})^2)\setminus (\mathcal G_k + [0,L_k)^2))$. 
This is the part of $\Delta_i$ which ``touches'' the boxes of side length $r_kL_k$ deleted from $\ballZ(0,2R)$ (more specifically, from $(\mathcal G_{k+1} + [0,L_{k+1})^2)$)
in the definition of $\mathbf G$, i.e., when defining $\mathcal G_k$. 
Let $N_k$ be the total number of such ``touched'' boxes of side length $r_kL_k$. 
Since each of these boxes has boundary $\leq 4r_k\frac{L_k}{L_0}$, it follows that $|E_k| \leq N_k\cdot16 r_k\frac{L_k}{L_0}$. 
Consider separately the cases $N_k> 27$ and $N_k\leq 27$. 
If $N_k\leq 27$, then 
\[
|E_k| \leq N_k\cdot 16 r_k\frac{L_k}{L_0}\leq 16\cdot 27\cdot \frac{r_k}{l_k}\cdot \frac{L_r}{L_0}
\leq 
C\cdot\frac{r_k}{l_k}\cdot |\Delta_i| ,\
\]
where the last inequality follows from \eqref{eq:lbDelta}. 
Assume now that $N_k>27$. 
From all these boxes we can choose $\geq \lceil \frac{N_k}{27}\rceil (\geq 2)$ boxes so that each pair of them is at $\ell^\infty$-distance $\geq L_{k+1}$ from each other.
By \cite[Lemma~2.1(ii)]{DeuschelPisztora}, the set $\{x\in\overline{\mathbf a}~:~\{x,y\}\in\Delta_i\mbox{ for some }y\}$ is $*$-connected. 
Thus, we can choose disjoint simple $*$-paths in $\{x\in\overline{\mathbf a}~:~\{x,y\}\in\Delta_i\mbox{ for some }y\}$ of $\frac{L_{k+1}}{3L_0}$ vertices each, originating near each of such boxes. 
(See Figure~\ref{fig:Deltas}.)
\begin{figure}[t]
\centering
\includegraphics[width=0.7\textwidth]{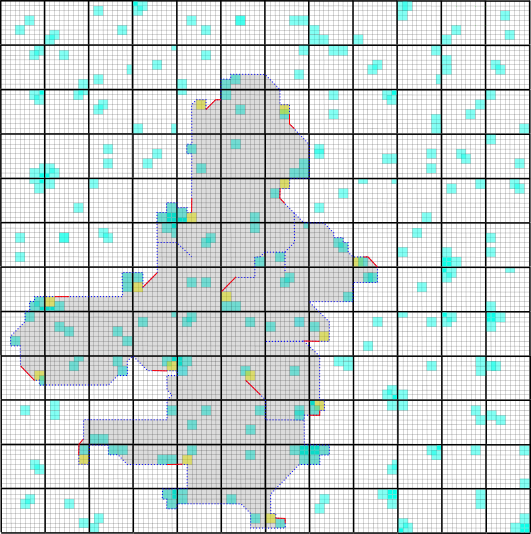}
\caption{The shaded region is $\overline{\mathbf b}_i$. Its boundary is the $*$-connected set $\{x\in\overline{\mathbf a}~:~\{x,y\}\in\Delta_i\mbox{ for some }y\}$.
The colored boxes correspond to $(\mathcal G_{k+1} + [0,L_{k+1})^2)\setminus (\mathcal G_k + [0,L_k)^2)$. 
The total number of boxes at pairwise distance $\geq L_{k+1}$ from each other is $\geq \lceil \frac{N_k}{27}\rceil (\geq 2)$. 
Each piece of the boundary of $\overline{\mathbf b}_i$ (illustrated with solid lines) touching one of the well-separated boxes consist of $\frac{L_{k+1}}{3L_0}$ vertices. 
Since these paths are disjoint, the total number of vertices in these paths is $\geq \lceil \frac{N_k}{27}\rceil\cdot \frac{L_{k+1}}{3L_0}$, which implies \eqref{eq:DeltaiNk}.
}\label{fig:Deltas}
\end{figure}
Therefore,  
\begin{equation}\label{eq:DeltaiNk}
|\Delta_i| \geq |\{x\in\overline{\mathbf a}~:~\{x,y\}\in\Delta_i\mbox{ for some }y\}|
\geq \frac 13\cdot\frac{L_{k+1}}{L_0}\cdot \frac{N_k}{27} ,\
\end{equation}
and we conclude that 
\[
|E_k| \leq N_k\cdot16 r_k\frac{L_k}{L_0}\leq 16\cdot 81\cdot \frac{r_k}{l_k}\cdot |\Delta_i| .\
\]
Combining the bounds of $|E_k|$ for all $k$ gives  
\[
|\delta_i| = |\Delta_i| - |\Delta_i\setminus\delta_i|
\geq |\Delta_i| - \sum_{k=0}^{r-1}|E_k|
\geq |\Delta_i|\cdot\left(1 - C\cdot\sum_{k=0}^{r-1}\frac{r_k}{l_k}\right) .\
\]
This is precisely \eqref{eq:relationboundaries}. Thus, the proof of Lemma~\ref{l:isopaj} is complete in the case $j=2$. 
\end{proof}

\bigskip
\bigskip

We proceed with the proof of Lemma~\ref{l:isopaj} in the case $j\geq 3$. 
\begin{proof}[Proof of Lemma~\ref{l:isopaj} ($j\geq 3$)]
The proof is by induction on $j$ and using the result of Lemma~\ref{l:isopaj} for $j=2$ proved before. 
Given $j\geq 3$, we assume that the statement of Lemma~\ref{l:isopaj} holds for all $j'<j$ and prove that it also holds for $j$.

Fix $x\in\GG_s\cap\ballZ(0,2R-3L_s)$, $y\in\GG_0\cap(x+[-L_s,2L_s)^d)$, $2\leq j\leq d$, and pairwise orthogonal $e_1,\dots,e_j\in\Z^d$ with $|e_i|_1=1$, 
and let $\mathbf g'$ be the restriction of $\mathbf g$ to $(y + \sum_{i=1}^j\Z\cdot e_i)$. 
Fix $\epsilon>0$ and a subset $\mathbf a'$ of $\mathbf g'$ with $|\mathbf a'|\in[\epsilon (\frac{3L_s}{L_0})^j, (1-\epsilon)(\frac{3L_s}{L_0})^j]$.

By Proposition~\ref{prop:Gproperties}(c), $\mathbf g'$ is connected and 
$|\mathbf g'| \geq f_j(\frac{r_0}{l_0})\cdot(\frac{3L_s}{L_0})^j$, where $f_j$ is defined in \eqref{eq:fj}.
There exists $\rho_3 = \rho_3(\epsilon,j)>0$ so that if $\frac{r_0}{l_0}<\rho_3$, then 
\begin{equation}\label{eq:r_0l_0isopaj3}
f_j\left(\frac{r_0}{l_0}\right)\geq f_{j-1}\left(\frac{r_0}{l_0}\right) > 1 - \frac{\epsilon}{2} ,\
\end{equation}
which implies that $|\mathbf g'| \geq (1 - \frac{\epsilon}{2})\cdot(\frac{3L_s}{L_0})^j$. 
(The inequality for $f_{j-1}(\frac{r_0}{l_0})$ is used later in the proof.)

\begin{figure}
\centering
\includegraphics[width=0.9\textwidth]{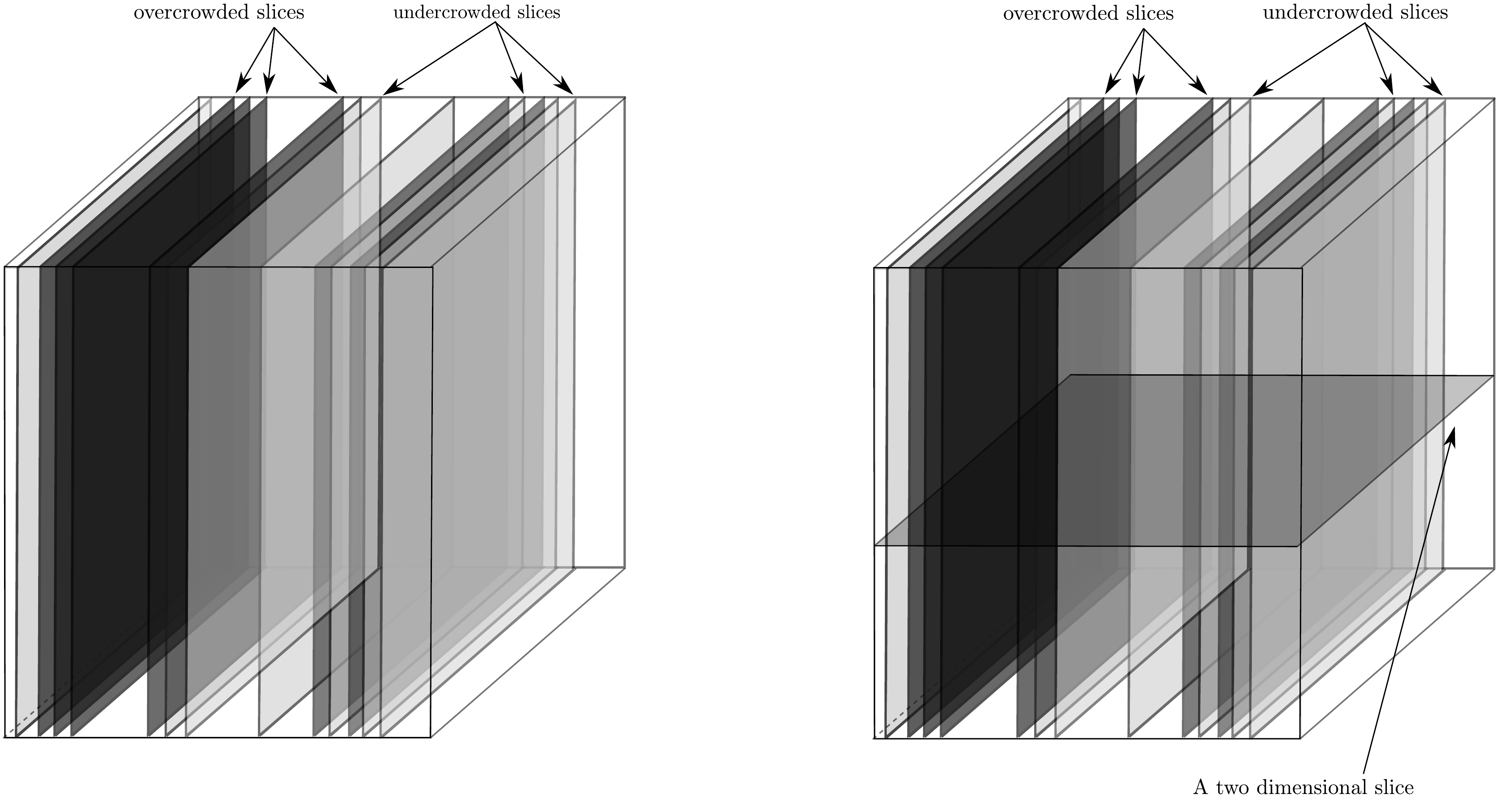}
\caption{An illustration of the overcrowded and undercrowded slices $M_k$. Every two-dimensional slice $N_t$ 
intersects each of the overcrowded and undercrowded slices in $\frac{3L_s}{L_0}$ vertices, 
so it intersects the union of all the overcrowded slices and also the union of all the undercrowded slices 
in $\frac{\epsilon}{8}\cdot(\frac{3L_s}{L_0})^2$ vertices.
}\label{fig:slices}
\end{figure}

Consider the $(j-1)$-dimensional slices $M_k = k\cdot L_0\cdot e_1 + (y + \sum_{i=2}^j\Z\cdot e_i)\cap(x + [-L_s,2L_s)^d)\cap\GG_0$, $k\in\Z$. 
Since $|\mathbf a'| \geq \epsilon (\frac{3L_s}{L_0})^j$, 
there exist at least $\frac{\epsilon}{2}\cdot(\frac{3L_s}{L_0})$ slices $M_k$ containing 
$\geq \frac{\epsilon}{2}\cdot(\frac{3L_s}{L_0})^{j-1}$ vertices from $\mathbf a'$. 
Since $|\mathbf g'\setminus \mathbf a'| \geq \frac{\epsilon}{2}\cdot (\frac{3L_s}{L_0})^j$,  
there exists at least $\frac{\epsilon}{4}\cdot(\frac{3L_s}{L_0})$ slices $M_k$ containing 
$\geq \frac{\epsilon}{4}\cdot(\frac{3L_s}{L_0})^{j-1}$ vertices from $\mathbf g'\setminus \mathbf a'$.

If there exists at least $\frac{\epsilon}{8}\cdot(\frac{3L_s}{L_0})$ slices $M_k$ containing 
$\geq \frac{\epsilon}{8}\cdot(\frac{3L_s}{L_0})^{j-1}$ vertices from {\it each} of the sets $\mathbf a'$ and $\mathbf g'\setminus \mathbf a'$, 
then the restriction of $\mathbf g'$ to any such slice satisfies the induction hypothesis. 
Therefore, by applying Lemma~\ref{l:isopaj} to the restriction of $\mathbf g'$ in each of these slices, we conclude that 
\[
|\partial_{\mathbf g'} \mathbf a'| \geq 
\gamma_{\scriptscriptstyle \ref{l:isopaj}}\left(\frac{\epsilon}{8},j-1\right)\cdot \left(\frac{3L_s}{L_0}\right)^{j-2}\cdot \frac{\epsilon}{8}\cdot\left(\frac{3L_s}{L_0}\right) 
= \gamma_{\scriptscriptstyle \ref{l:isopaj}}\left(\frac{\epsilon}{8},j-1\right)\cdot\frac{\epsilon}{8}\cdot\left(\frac{3L_s}{L_0}\right)^{j-1} .\
\]

If there exists $<\frac{\epsilon}{8}\cdot(\frac{3L_s}{L_0})$ slices $M_k$ containing 
$\geq \frac{\epsilon}{8}\cdot(\frac{3L_s}{L_0})^{j-1}$ vertices from each of the sets $\mathbf a'$ and $\mathbf g'\setminus \mathbf a'$, 
then by earlier conclusion, there exist at least $\frac{\epsilon}{8}\cdot(\frac{3L_s}{L_0})$ slices $M_k$ containing 
$<\frac{\epsilon}{8}\cdot(\frac{3L_s}{L_0})^{j-1}$ vertices from $\mathbf g'\setminus \mathbf a'$. 
By Proposition~\ref{prop:Gproperties}(c) and \eqref{eq:r_0l_0isopaj3}, each such slice contains at least 
$f_{j-1}(\frac{r_0}{l_0})\cdot(\frac{3L_s}{L_0})^{j-1}\geq (1 - \frac {\epsilon}{2})\cdot (\frac{3L_s}{L_0})^{j-1}$ vertices from $\mathbf g'$. 
Therefore, there exist at least $\frac{\epsilon}{8}\cdot(\frac{3L_s}{L_0})$ slices $M_k$ containing 
$\geq (1-\frac{5\epsilon}{8})\cdot(\frac{3L_s}{L_0})^{j-1}$ vertices from $\mathbf a'$. 
We choose $\frac{\epsilon}{8}\cdot(\frac{3L_s}{L_0})$ of them and call these slices {\it overcrowded}. 

Similarly one shows that there exist at least $\frac{\epsilon}{8}\cdot(\frac{3L_s}{L_0})$ slices $M_k$ containing 
$\geq (1-\frac{5\epsilon}{8})\cdot(\frac{3L_s}{L_0})^{j-1}$ vertices from $\mathbf g'\setminus \mathbf a'$. 
We choose $\frac{\epsilon}{8}\cdot(\frac{3L_s}{L_0})$ of them and call these slices {\it undercrowded}. 

Consider now the two-dimensional slices 
$N_t = \sum_{i=3}^jt_i\cdot L_0\cdot e_i + (y + \Z\cdot e_1 + \Z\cdot e_2)\cap(x + [-L_s,2L_s)^d)\cap\GG_0$, $t = (t_3,\dots,t_j)\in\Z^{j-2}$, see Figure~\ref{fig:slices}.
Note that every non-empty two-dimensional slice $N_t$ intersects 
the union of all the overcrowded $(j-1)$-dimensional slices and also the union of all the undercrowded $(j-1)$-dimensional slices 
in $\frac{\epsilon}{8}\cdot(\frac{3L_s}{L_0})^2$ vertices. 
Therefore, there exist at least $\frac 23\cdot(\frac{3L_s}{L_0})^{j-2}$ two-dimensional slices $N_t$ containing at least 
$\frac{\epsilon}{16}\cdot(\frac{3L_s}{L_0})^2$ vertices from $\mathbf a'$, and 
at least $\frac 23\cdot(\frac{3L_s}{L_0})^{j-2}$ slices $N_t$ containing at least 
$\frac{\epsilon}{16}\cdot(\frac{3L_s}{L_0})^2$ vertices from $\mathbf g'\setminus \mathbf a'$. 
Indeed, assume that there exist ${\mathrm N}< \frac 23\cdot(\frac{3L_s}{L_0})^{j-2}$ slices containing at least 
$\frac{\epsilon}{16}\cdot(\frac{3L_s}{L_0})^2$ vertices from $\mathbf a'$. 
The other case is considered similarly. 
Then the total number of vertices of $\mathbf a'$ in overcrowded $(j-1)$-dimensional slices is 
\begin{multline*}
\leq {\mathrm N}\cdot\left(\frac{\epsilon}{8}\cdot\left(\frac{3L_s}{L_0}\right)^{2}\right) 
+ \left(\left(\frac{3L_s}{L_0}\right)^{j-2} - {\mathrm N}\right)\cdot\left(\frac{\epsilon}{16}\cdot\left(\frac{3L_s}{L_0}\right)^{2}\right)\\
< \frac 53\cdot \frac{\epsilon}{16}\cdot\left(\frac{3L_s}{L_0}\right)^j
< \left(1-\frac{5\epsilon}{8}\right)\cdot \frac{\epsilon}{8}\cdot\left(\frac{3L_s}{L_0}\right)^j ,\
\end{multline*}
for $\epsilon<\frac{4}{15}$, which contradicts with the definition of overcrowded slices. 

Therefore, there exist at least $\frac 13\cdot(\frac{3L_s}{L_0})^{j-2}$ slices $N_t$ containing at least 
$\frac{\epsilon}{16}\cdot(\frac{3L_s}{L_0})^2$ vertices from $\mathbf a'$ and 
at least $\frac{\epsilon}{16}\cdot(\frac{3L_s}{L_0})^2$ vertices from $\mathbf g'\setminus \mathbf a'$. 
We can now apply Lemma~\ref{l:isopaj} to each of such slices to obtain that the boundary of $\mathbf a'$ 
in the restriction of $\mathbf g'$ to each of such slices is at least $\gamma_{\scriptscriptstyle \ref{l:isopaj}}(\frac{\epsilon}{16},2)\cdot \frac{3L_s}{L_0}$. 
Since the total number of slices is at least $\frac 13\cdot (\frac{3L_s}{L_0})^{j-2}$, 
we conclude that 
\[
|\partial_{\mathbf g'} \mathbf a'| \geq 
\frac 13\cdot \left(\frac{3L_s}{L_0}\right)^{j-2}\cdot \gamma_{\scriptscriptstyle \ref{l:isopaj}}\left(\frac{\epsilon}{16},2\right)\cdot \frac{3L_s}{L_0}
=
\frac 13\cdot \gamma_{\scriptscriptstyle \ref{l:isopaj}}\left(\frac{\epsilon}{16},2\right)\cdot \left(\frac{3L_s}{L_0}\right)^{j-1} .\
\]
Thus the result of Lemma~\ref{l:isopaj} for the given $j$ follows with the choice of 
\[
\gamma_{\scriptscriptstyle \ref{l:isopaj}}(\epsilon,j) = 
\min\left(\frac{\epsilon}{8}\cdot \gamma_{\scriptscriptstyle \ref{l:isopaj}}\left(\frac{\epsilon}{8},j-1\right), 
\frac 13\cdot \gamma_{\scriptscriptstyle \ref{l:isopaj}}\left(\frac{\epsilon}{16},2\right)\right)
\]
and
\[
\rho_{\scriptscriptstyle \ref{l:isopaj}}(\epsilon,j) = 
\min\left(\rho_3(\epsilon,j),\rho_{\scriptscriptstyle \ref{l:isopaj}}\left(\frac{\epsilon}{8},j-1\right),\rho_{\scriptscriptstyle \ref{l:isopaj}}\left(\frac{\epsilon}{16},2\right)\right) .\
\]
The proof of Lemma~\ref{l:isopaj} in the case $j\geq 3$ is complete. 
\end{proof}

\section{Quenched invariance principle}\label{sec:ipa}

In this section we state the quenched invariance principle for simple random walk on percolation clusters satisfying some general conditions. 
Later, in Section~\ref{sec:ip}, we show that these conditions are satisfied by any probability measure $\mathbb P^u$, for $u\in(a,b)$, 
given that the family $\{\mathbb P^u\}_{u\in(a,b)}$ satisfies the axioms \p{} -- \ppp{} and \s{} -- \sss{}.

Consider a probability measure $\mathbb P$ on the measurable space $(\Omega,\mathcal F)$, 
where $\Omega = \{0,1\}^{\Z^d}$, $d\geq 2$, and $\mathcal F$ is the sigma-algebra 
generated by the canonical coordinate maps $\{\omega\mapsto\omega(x)\}_{x\in\Z^d}$. 
For $x\in\Z^d$, denote by $\tau_x~:~\Omega\to\Omega$ the shift in direction $x$, i.e., 
$(\tau_x\omega)(y) = \omega(x+y)$. For each $\omega\in\Omega$, let 
\[
\set = \set(\omega) = \{x\in\Z^d~:~\omega(x) = 1\} .\
\]
We think about $\set$ as a subgraph of $\Z^d$ in which edges are added between any two vertices of $\set$ of $\ell^1$-distance $1$. 
As before, we denote by $\set_\infty$ the subset of vertices of $\set$ which belong to infinite connected components of $\set$. 
We assume that $\mathbb P$ satisfies the following axioms. 
\begin{itemize}
\item[\a{}]
For all $e\in\Z^d$ with $|e|_1 = 1$, the shift $\tau_e$ is measure preserving and ergodic on $(\Omega,\mathcal F, \mathbb P)$.  
\item[\aa{}] 
The subgraph $\set_\infty$ is non-empty and connected, $\mathbb P$-a.s. (In particular, $\mathbb P[0\in\set_\infty] > 0$.)
\item[\aaa{}]
There exist constants $c>0$, $C<\infty$, and $\Delta_3>0$ such that for all $R\geq 1$ and for all $e\in\Z^d$ with $|e|_1 = 1$, 
\[
\mathbb P\left[k\cdot e\in \set_\infty\mbox{ for some }0\leq k\leq R \right]\geq 1-C\cdot e^{-c(\log R)^{1+\Delta_3}} .\
\]
\end{itemize}
Our next axioms on $\mathbb P$ concern intrinsic geometry of $\set_\infty$. 
For $x,y\in\set$, let $\rho_\set(x,y)\in \mathbb N \cup \{\infty\}$ denote the distance between $x$ and $y$ in $\set$, i.e.,
\[
\rho_\set(x,y)=\inf\left\{ n\geq 0 ~:~
    \begin{array}{c}
        \text{there exist $x_0,\dots,x_n\in\set$ such that}\\
\text{$x_0=x$, $x_n=y$, and}\\
        \text{$|x_k-x_{k-1}|_1=1$ for all $k=1,\dots,n$}
    \end{array}
    \right\},
\]
where we use the convention $\inf\emptyset=\infty$, and let $\ballS(x,R) = \{y\in\set~:~\rho_\set(x,y)\leq R\}$. 
\begin{itemize}
\item[\aaaa{}] 
There exist constants $c>0$, $C<\infty$, and $\Delta_4>0$ such that for all $R\geq 1$, 
\[
\mathbb P\left[\text{for all }x,y\in \set_\infty\cap \ballZ(0,R), ~\rho_\set(x,y)\leq C\cdot R \right]\geq 1-C\cdot e^{-c(\log R)^{1+\Delta_4}} .\
\]
\item[\aaaaa{}]
$\mathbb P[\cdot~|~0\in\set_\infty]$-almost surely, 
\[
\inf_{k\geq 1}\inf\left\{
\frac{|\partial_\set A|}{|A|^{\frac{d-1}{d}}}~:~A\subset \set_\infty\cap\ballS(0,2k), |A|\geq k^{1/3}
\right\}>0 .\
\]
\end{itemize}

\bigskip

Next we describe the random walk on $\set$. 
For $\omega\in\Omega$ and $x\in\set$, let $\deg_\omega(x)=\left|\{y\in\set ~:~ |y-x|_1=1 \}\right|$ be the degree of $x$ in $\set$.
For a configuration $\omega\in\Omega$ and $x\in\set$, let $\mathbf P_{\omega,x}$ be 
the distribution of the random walk $\{X_n\}_{n \geq 0}$ on $\set$ defined by the transition kernel
\begin{equation}\label{eq:trans}
\mathbf P_{\omega,x}[X_{n+1} = z|X_{n}=y] =
\left\{
    \begin{array}{ll}
        \frac{1}{2d}&\quad |z-y|_1=1,~z\in\set;\\
        1-\frac{\deg_\omega(y)}{2d}&\quad z=y;\\
        0 & \quad \text{otherwise,}
    \end{array}
    \right.
\end{equation}
and initial position $\mathbf P_{\omega,x}[X_0=x]=1$. The corresponding expectation is denoted by $\mathbf E_{\omega,x}$. 

Let $\Omega_0=\left\{\omega\in\Omega ~:~ 0\in\set_\infty\right\}$, and 
define the measure $\mathbb P_0$ by $\mathbb P_0[A]=\mathbb P[A~|~\Omega_0]$. 
We denote by $\mathbb E_0$ the expectation with respect to $\mathbb P_0$. 

For $\omega\in\Omega_0$, $n\in \mathbb N$, and $t\geq 0$, let 
\[
\widetilde B_n(t) = \frac{1}{\sqrt n}\left(X_{\lfloor tn\rfloor} + \left(tn - \lfloor tn\rfloor\right)\cdot \left(X_{\lfloor tn\rfloor + 1} - X_{\lfloor tn\rfloor}\right)\right) ,\
\]
where $(X_k)_{k\geq 0}$ is the random walk on $\set$ (actually on $\set_\infty$) with distribution $\mathbf P_{\omega,0}$. 
Theorem~\ref{thm:ip} follows from Theorem~\ref{thm:ipa}, as we demonstrate in Section~\ref{sec:ip}. 
\begin{theorem}\label{thm:ipa}
Let $d\geq 2$, and assume that the measure $\mathbb P$ satisfies assumptions \a{} -- \aaaaa{}. 
Then for all $T>0$ and for $\mathbb P_0$-almost every $\omega$, the law of $(\widetilde B_n(t))_{0\leq t\leq T}$ on $(C[0,T],\mathcal W_T)$ converges weakly to the law of 
a Brownian motion with zero drift and non-degenerate covariance matrix. 
In addition, if reflections and rotations of $\Z^d$ by $\frac{\pi}{2}$ preserve $\mathbb P$, 
then the limiting Brownian motion is isotropic (with positive diffusion constant).
\end{theorem}
The proof of Theorem~\ref{thm:ipa} is a routine adaptation of the proof of \cite[Theorem~1.1]{BergerBiskup}. 
Instead of proving \cite[Theorem~6.3]{BergerBiskup}, which relies on the upper bound on heat kernel obtained in \cite[Theorem~1]{Barlow}, 
we follow the proof of \cite[Theorem~2.1]{BiskupPrescott}, which uses softer arguments (still relying very much on observations from 
\cite{Bass,Nash} exploited in \cite{Barlow}, but not 
using the full strenth of the upper bound in \cite[Theorem~1]{Barlow}).
We give a sketch proof of Theorem~\ref{thm:ipa} in Section~\ref{sec:ipaproof}.

\section{Proof of Theorem~\ref{thm:ip}}\label{sec:ip}

In this section we derive Theorem~\ref{thm:ip} from Theorem~\ref{thm:ipa}. 
Namely, we prove that for a family of probability measures $\{\mathbb P^u\}_{u\in(a,b)}$ satisfying \p{} -- \ppp{} and \s{} -- \sss{}, 
every probability measure $\mathbb P^u$ in the family satisfies the conditions \a{} -- \aaaaa{} of Section~\ref{sec:ipa}. 
Our proof can mostly be read independently of Sections~\ref{sec:renormalization} and \ref{sec:isoperimetric}, 
except for the proof of \aaa{}, where we need to use and generalize some results from Section~\ref{sec:renormalization}. 
Fix $u\in(a,b)$. We prove that $\mathbb P^u$ satisfies \a{} -- \aaaaa{}. 

\medskip

$\bullet$ Condition \a{} follows from \p{}.

\medskip

$\bullet$ Condition \aa{} follows from \s{} -- \sss{}.

\medskip

$\bullet$ The fact that \aaaa{} follows from \p{} -- \ppp{} and \s{} -- \sss{} is proved in \cite[Theorem~1.3]{DRS12}.

\medskip

$\bullet$ Condition \aaaaa{} follows from Theorem~\ref{thm:isoperimetric}, \p{} (only translation invariance part), and \s{}. 
It suffices to show that for $\mathbb P^u[\cdot~|~0\in\set_\infty]$-almost every realization $\atom$ and all $R$ sufficiently large, 
the connected component of $0$ in $\set_\infty\cap\ballZ(0,R)$ is the unique largest in volume connected component of $\set\cap\ballZ(0,R)$, 
i.e., using the notation of Theorem~\ref{thm:isoperimetric}, $0\in\mathcal C_R$. 
Indeed, as soon as $0\in\mathcal C_R$ for all large $R$, the inclusion $\ballS(0,R)\subset\mathcal C_R$ holds for all large $R$, and 
\aaaaa{} follows from Theorem~\ref{thm:isoperimetric} and the Borel-Cantelli lemma. 

To prove the remaining claim, we apply \s{} to all the boxes $\ballZ(x,R^{1/2d})$, $x\in \ballZ(0,R-4R^{1/2d})$, 
(this is possible by \p{}) and use the Borel-Cantelli lemma to conclude that 
$\mathbb P^u$-almost surely for all large $R$, 
\begin{itemize}\itemsep0pt
\item[(a)]
each box $\ballZ(x,R^{1/2d})$, $x\in \ballZ(0,R-4R^{1/2d})$, intersects $\set_{R^{1/2d}}$,
\item[(b)] 
for any $x,x'\in \ballZ(0,R-4R^{1/2d})$ such that $|x-x'|_1 = 1$, there exists a unique connected component of $\set_{R^{1/2d}}\cap\ballZ(x,4R^{1/2d})$ 
which intersects $\ballZ(x,R^{1/2d})\cup\ballZ(x',R^{1/2d})$. 
\end{itemize}
Statements (a) and (b) together imply that 
$\mathbb P^u$-almost surely for all large $R$, 
there exists a connected component of $\set_{R^{1/2d}}\cap\ballZ(0,R)$ which intersects every box $\ballZ(x,R^{1/2d})$, $x\in \ballZ(0,R-4R^{1/2d})$, 
and it is the unique connected component of $\set_{R^{1/2d}}\cap\ballZ(0,R)$ which intersects $\ballZ(0,R - 3R^{1/2d})$. 
Therefore, $\mathbb P^u[\cdot~|~0\in\set_\infty]$-almost surely for all large $R$, 
(1) the connected component of $0$ in $\set_\infty\cap\ballZ(0,R)$ intersects each box $\ballZ(x,R^{1/2d})$, $x\in \ballZ(0,R-4R^{1/2d})$, 
and (2) it is the unique connected component of $\set_{R^{1/3}}\cap\ballZ(0,R)$ which intersects $\ballZ(0,R - R^{\frac 13})$. 
By (1), the connected component of $0$ in $\set_\infty\cap\ballZ(0,R)$ has volume $\geq \frac 12 R^{d-\frac 12}$. 
Note that for all large $R$, any connected component of $\set\cap\ballZ(0,R)$ with volume $\geq \frac 12 R^{d-\frac 12}$ has diameter $\geq R^{\frac 13}$ 
and intersects $\ballZ(0,R - R^{\frac 13})$. 
By (2), such connected component must be unique. This implies the claim, and \aaaaa{} follows.

\medskip

$\bullet$ It remains to show that \aaa{} follows from \p{} -- \ppp{} and \s{} -- \sss{}. 
This is done by exploiting the renormalization structure of \cite{DRS12} and adding an additional increasing event to the structure. 
More precisely, we modify Definition~\ref{def:Axu} of event $A^u_x$. 
Let $\{e_i\}_{i=1}^d$ be the unit coordinate vectors in $\Z^d$. 
For $x\in \GG_0$ and $u\in(a,b)$, let $\mathcal A^u_x\in\mathcal F$ be the event that
\begin{itemize}\itemsep0pt
\item[(a)]
for each $e\in\{0,1\}^d$, the set $\set_{L_0}\cap(x+eL_0 + [0,L_0)^d)$
contains a connected component with at least $\frac 34 \eta(u) L_0^d$ vertices,
\item[(b)]
all of these $2^d$ components are connected in 
$\set \cap(x+[0,2L_0)^d)$, 
\item[(c)]
for each $1\leq i\leq d$, 
the ``special'' connected component of $\set_{L_0}$ in $(x + [0,L_0)^d)$ contains a vertex in each of the $d$ line segments 
$I_{x,i} = (x+(\lfloor \frac{L_0}{2}\rfloor, \dots, \lfloor \frac{L_0}{2}\rfloor) + \Z\cdot e_i)\cap(x + [\lfloor \frac{L_0}{3}\rfloor,\lfloor \frac{2L_0}{3}\rfloor )^d)$.  
\end{itemize}
For $u\in(a,b)$ and $x\in \GG_0$, let 
$\overline {\mathcal A}^u_{x,0}$ be the complement of $\mathcal A^u_{x}$, 
and for $u\in(a,b)$, $k\geq 1$, and $x\in\GG_k$ define inductively
\[
\overline {\mathcal A}^u_{x,k} = 
\bigcup_{\begin{array}{c}\scriptscriptstyle{x_1,x_2\in \GG_{k-1}\cap(x + [0,L_k)^d)} \\ \scriptscriptstyle{|x_1-x_2|_\infty \geq r_{k-1} \cdot L_{k-1}}\end{array}}
  \overline {\mathcal A}^u_{x_1,k-1} \cap \overline {\mathcal A}^u_{x_2,k-1}   \; .\
\]
By \p{} and Birkhoff's ergodic theorem, for any $u\in(a,b)$, $x\in\GG_0$, and $1\leq i\leq d$, 
\[
\lim_{L_0 \to \infty} \frac{3}{L_0} 
\sum_{ y \in I_{x,i}} \mathds{1}_{\{y \in\set_{L_0}\}}
    \; \stackrel{\mathbb{P}^u\text{-a.s.}}{=} \; 
\lim_{L_0 \to \infty} \frac{3}{L_0} 
\sum_{ y \in I_{x,i}} \mathds{1}_{\{y \in\set_\infty\}}
    \; \stackrel{\mathbb{P}^u\text{-a.s.}}{=} \; \eta(u) .\
\]
We conclude from \s{}, \sss{}, and \cite[(4.3)]{DRS12} that 
for any $u\in(a,b)$ there exists $\delta = \delta(u)>0$ such that $(1-\delta)u>a$ and 
\[
\mathbb P^{(1-\delta)u}\left[\mathcal A_0^u\right] \to 1 ,\quad \mbox{as} \quad L_0\to\infty .\
\]
As in the proof of \cite[Lemma~4.2]{DRS12}, this implies that 
for each $u\in(a,b)$, there exist $C = C(u)<\infty$ and $C' = C'(u,l_0)<\infty$ such that for all 
$l_0,r_0\geq C$, $L_0\geq C'$, and $k\geq 0$, 
\begin{equation}\label{eq:calA}
\mathbb P^u\left[\overline {\mathcal A}^u_{0,k}\right] \leq 2^{-2^k} .\
\end{equation}
We modify Definition~\ref{def:good} by replacing the events $A^u_x$ by $\mathcal A^u_x$. 
Let $u\in(a,b)$. For $k\geq 0$, we say that $x\in\GG_k$ is $k$-{\it bad} if the event 
$\overline {\mathcal A}^u_{x,k}\cup \overline B^u_{x,k}$
occurs, where $B^u_{x,k}$ is defined in \eqref{def:Buxk}. Otherwise, we say that $x$ is $k$-{\it good}. 
It follows from \eqref{eq:calA} and \cite[Lemma~4.4]{DRS12} that
for each $u\in(a,b)$, there exist $C = C(u)<\infty$ and $C' = C'(u,l_0)<\infty$ such that for all 
$l_0,r_0\geq C$, $L_0\geq C'$, and $k\geq 0$, 
\begin{equation}\label{eq:kbad:new}
\mathbb P^u\left[0\mbox{ is }k\mbox{-bad}\right] \leq 2\cdot 2^{-2^k} .\
\end{equation}
Note that if $0$ is $k$-good, then $\GG_0\cap[0,L_k)^d$ contains a connected component $\mathcal G$ of 
$0$-good vertices of diameter $\frac{L_k}{L_0}$ (in $\GG_0$)
which intersects every line segment $\Z\cdot e_i\cap[0,L_k)^d$, $1\leq i\leq d$. 
This is easily proved by induction from the definition of $k$-good vertex. 
By Lemma~\ref{l:fromG0toZd} and noting that any $0$-good vertex in the new sense is also $0$-good in the sense of Definition~\ref{def:good}, 
the set $\cup_{x\in\mathcal G}\mathcal C_x$ is contained in the same connected component of $\set$ with diameter at least $\frac{L_k}{2}$.
($\mathcal C_x$ is the ``special'' component of $\set\cap(x + [0,L_0)^d)$ defined in Lemma~\ref{l:fromG0toZd}(a).) 
By \s{} and \eqref{eq:C1:infty}, with probability $\geq 1 - Ce^{-c(\log L_k)^{1+\constS}} - 2\cdot 2^{-2^k}$, 
$\cup_{x\in\mathcal G}\mathcal C_x\subset \set_\infty$. 
By the definition of $0$-good vertex, namely using part (c) in the definition of $\mathcal A^u_x$, we obtain that 
\[
\mathbb P^u\left[\set_\infty\cap\left(\left(\left\lfloor \frac{L_0}{2}\right\rfloor, \dots, \left\lfloor \frac{L_0}{2}\right\rfloor\right) + \Z\cdot e_i\right)\cap[0,L_k)^d\neq\emptyset\right]
\geq 1 - Ce^{-c(\log L_k)^{1+\constS}} - 2\cdot 2^{-2^k} .\
\]
For $R\geq 1$, choose the largest $k$ such that $L_k\leq R$. Then as in \eqref{eq:Ls:lowerbound} and \eqref{eq:s:bounds}, we obtain that 
$\log L_k \geq c\log R$ and $2^k \geq (\log R)^{1+\constS}$ for all $R$ large enough. This implies that 
\begin{equation}\label{eq:aaa}
\mathbb P^u\left[\set_\infty\cap\left(\left(\left\lfloor \frac{L_0}{2}\right\rfloor, \dots, \left\lfloor \frac{L_0}{2}\right\rfloor\right) + \Z\cdot e_i\right)\cap\ballZ(0,R)\neq\emptyset\right]
\geq 1 - Ce^{-c(\log R)^{1+\constS}} .\
\end{equation}
Assumption \aaa{} now follows from \p{} and \eqref{eq:aaa}.

\medskip

We have checked that for any $u\in(a,b)$, the probability measure $\mathbb P^u$ satisfies the assumptions \a{} -- \aaaaa{}, 
given that the family $\{\mathbb P^u\}_{u\in(a,b)}$ satisfies \p{} -- \ppp{} and \s{} -- \sss{}. 
Thus, Theorem~\ref{thm:ip} follows from Theorem~\ref{thm:ipa}.
\qed

\bigskip

\section{Remarks on ergodicity assumption}\label{sec:ergodicityremarks}

In this section we discuss possible weakenings of assumption \p{}, more precisely, its part concerning with ergodicity of $\mathbb P^u$. 
Condition \p{} requires ergodicity of $\mathbb P^u$ with respect to {\it every} shift of $\Z^d$, i.e., 
$\mathbb P^u[E]\in\{0,1\}$ for every $E\in\mathcal F$ such that $\tau_x(E) = E$ for {\it some} $x\in\Z^d$.  
This is crucially used in the proof of the shape theorem in \cite{DRS12}. 
However, the proof of \cite[Theorem~1.3]{DRS12} goes through under the milder assumption of ergodicity of $\mathbb P^u$ with respect to the group $\Z^d$, 
i.e., $\mathbb P^u[E]\in\{0,1\}$ for every $E\in\mathcal F$ such that $\tau_x(E) = E$ for {\it all} $x\in\Z^d$.
Indeed, the only place where ergodicity is used in the proof of \cite[Theorem~1.3]{DRS12} is \cite[(4.1)]{DRS12}, which still holds under the weaker assumption. 
Since \cite[(4.1)]{DRS12} is used in the proof of Lemma~\ref{l:Guxk} (Lemmas~4.2 and 4.4 in \cite{DRS12}), 
and since we do not use any form of ergodicity of $\mathbb P^u$ elsewhere in the proof of 
Theorem~\ref{thm:isoperimetric}, we conclude that the result of Theorem~\ref{thm:isoperimetric} holds even if we replace the ergodicity 
of $\mathbb P^u$ with respect to every shift of $\Z^d$ in \p{} by the ergodicity of $\mathbb P^u$ with respect to the group $\Z^d$.

Similarly, in the proof of the quenched invariance principle we do not need the full strength of assumption \p{}. 
Apart from the proof of Theorem~\ref{thm:isoperimetric}, we use ergodicity of $\mathbb P^u$ to check assumptions \a{}, \aaa{}, and \aaaa{}. 
Assumptions \a{} and \aaa{} hold under the milder assumption of ergodicity of $\mathbb P^u$ with respect to each shift along a coordinate direction, 
i.e., $\mathbb P^u[E]\in\{0,1\}$ for every $E\in\mathcal F$ such that $\tau_e(E) = E$ for some $e\in\Z^d$ with $|e|_1 = 1$. 
Assumption \aaaa{} holds under assumption of ergodicity of $\mathbb P^u$ with respect to the group $\Z^d$, as discussed just above. 
Therefore, the result of Theorem~\ref{thm:ip} holds when the ergodicity 
of $\mathbb P^u$ with respect to every shift of $\Z^d$ in \p{} is replaced by 
the ergodicity of $\mathbb P^u$ with respect to each shift along a coordinate direction of $\Z^d$.
We remark that in the case of the random conductance model with elliptic coefficients, the quenched invariance principle holds 
under the ergodicity of random coefficients with respect to the group $\Z^d$ and some moment assumptions, see \cite{ADS13,Biskup}. 
The tricky part is discussed at the end of the proof of \cite[Lemma~4.8]{Biskup}. 
It crucially relies on the positivity of all the coefficients (every vertex of $\Z^d$ can be visited by the random walk) 
and does not generally apply if some coefficients are $0$.

\bigskip

\appendix

\section{Proof of Theorem~\ref{thm:ipa}}\label{sec:ipaproof}

The proof of Theorem~\ref{thm:ipa} follows closely the proof of \cite[Theorem~1.1]{BergerBiskup} 
(see \cite[Section~1.4]{BergerBiskup} there for an outline of the proof) using essential simplifications obtained in \cite{BiskupPrescott}. 
Therefore, we only give a brief sketch here. 
Let us point out that our model fits well into the setup of \cite{BiskupPrescott}, with conductances taking values in $\{0,1\}$. 
In particular, in our situation, there is no need for the truncation argument used in \cite{BiskupPrescott}, since 
we may assume that $\alpha$ defined in \cite[(2.8) and (2.9)]{BiskupPrescott} equals $1$. 
Therefore, using the notation of \cite{BiskupPrescott}, $\mathcal C_{\infty,\alpha} = \mathcal C_\infty$, 
and many arguments of \cite{BiskupPrescott} simplify in our setting. 
(In particular, the continous time random walk $\mathcal Y_t$ defined in \cite[(2.14)]{BiskupPrescott} makes only nearest neighbor jumps in $\mathcal C_\infty$.)
\begin{proof}[Proof of Theorem~\ref{thm:ipa}]
In the heart of the proof is the following result, see \cite[Theorem~2.2]{BergerBiskup}: 
If $\mathbb P$ satisfies \a{} and \aa{}, then there exists a function $\chi~:~\Z^d\times\Omega_0 \to \R^d$ defined by \cite[(2.11)]{BergerBiskup} such that 
\begin{enumerate}\itemsep0pt
\item
for every $x\in\Z^d$, $\chi(x,\cdot)\in L^2(\Omega,\mathcal F,\mathbb P_0)$;
\item
for every $\omega\in\Omega_0$, $\chi(0,\omega) = 0$;
\item
for $\mathbb P_0$-almost every $\omega\in\Omega_0$, and for all $x,y\in\set_\infty$, $\chi(x,\omega) - \chi(y,\omega) = \chi(x-y,\tau_y(\omega))$;
\item
for $\mathbb P_0$-almost every $\omega\in\Omega_0$, the function $x\mapsto x + \chi(x,\omega)$ is harmonic with respect to transition probabilities \eqref{eq:trans};
\item
there exists $C<\infty$ such that for all $x,e\in\Z^d$ with $|e|_1 = 1$, 
\[
\|(\chi(x+e,\cdot) - \chi(x,\cdot))\mathds{1}_{\{x\in\set_\infty\}}\mathds{1}_{\{\omega(x) =1,~ \omega(x+e) = 1\}}\|_2< C .\
\]
\end{enumerate}
The function $\chi$ is called a {\it corrector}. 
Classically (see, e.g., the proof of \cite[Theorem~2.1]{BiskupPrescott} or \cite[Theorem~1.1]{BergerBiskup} in the case $d=2$), 
in order to prove convergence to Brownian motion in Theorem~\ref{thm:ipa}, it suffices to prove that the corrector is sublinear in the following sense.
\begin{lemma}\label{l:chimaxsublinearity}
Let $d\geq 2$, and $\mathbb P$ satisfies \a{}--\aaaaa{}. Then for $\mathbb P_0$-almost every $\omega$, 
\begin{equation}\label{eq:chimaxsublinearity}
\lim_{k\to\infty}\max_{x\in\set_\infty\cap\ballZ(0,k)}\frac{|\chi(x,\omega)|}{k} = 0 .\
\end{equation}
\end{lemma}
Before we give the proof of Lemma~\ref{l:chimaxsublinearity}, we finish the proof of Theorem~\ref{thm:ipa}. 
As already indicated earlier, the proof of convergence to Brownian motion with zero drift follows line by line the proof of \cite[Theorem~1.1]{BergerBiskup} in the case $d=2$, 
see also the proof of \cite[Theorem~2.1]{BiskupPrescott}. 
The statements about symmetric $\mathbb P$ can be treated also as in the proof of \cite[Theorem~2.1]{BiskupPrescott}.
The fact that the covariance matrix of the limiting Brownian motion is non-degenerate follows from the sublinearity of the corrector, similarly to \cite{BiskupPrescott}. 
However, since the proof of the analogous fact in \cite{BiskupPrescott} benefits from the rotational and reflectional symmetries of the measure, 
which we do not assume here, we present some details below. 

Similarly to \cite{BiskupPrescott}, we define for $\omega\in\Omega$ and $x\in\Z^d$, the function $\varphi_\omega(x) = x + \chi(x,\omega)$. 
Let $M = \varphi_\omega(X_1)$, where $X_1$ is the first step of the random walk defined in \eqref{eq:trans}. 
As in \cite[(5.30) and (5.31)]{BiskupPrescott}, the (deterministic) covariance matrix $\Sigma^2$ of the limiting Brownian motion satisfies 
\[
\langle v, \Sigma^2 v\rangle = \mathbb E_0 \mathbf E_{\omega,0} \left[\langle v, M\rangle^2\right], \qquad\text{for all $v\in\mathbb R^d$.}
\]
Assume that $\langle v, \Sigma^2 v\rangle = 0$ for some $v\in\mathbb R^d$. 
Since $\set_\infty$ is $\mathbb P$-almost surely connected, $\mathbb P_0[x\in \set_\infty]>0$ for all $x\in\mathbb Z^d$ with $|x|_1 = 1$. 
Thus, for each such $x$, $\mathbb E_0 [\langle v,\varphi_\omega(x)\rangle^2\cdot\mathds{1}_{\{x\in\set_\infty\}}] = 0$, which implies that 
\begin{equation}\label{eq:scalprod=0:1}
\text{$\mathbb E\left[|\langle v,\varphi_\omega(x)\rangle|\cdot \mathds{1}_{\{0,x\in\set_\infty\}}(\omega)\right] = 0$, for all $x\in\Z^d$ with $|x|_1=1$.}
\end{equation}
We will prove that 
\begin{equation}\label{eq:scalprod=0:2}
\text{$\langle v, \varphi_\omega(y)\rangle\cdot \mathds{1}_{\{y\in\set_\infty\}} = 0$ for all $y\in\Z^d$ and $\mathbb P_0$-almost every $\omega$.}
\end{equation}
Let $x_0=0,x_1,\ldots,x_n$ be a simple nearest neighbor path in $\Z^d$. Then,
\begin{eqnarray*}
\mathbb E\left[|\langle v,\varphi_\omega(x_n)\rangle|\cdot \mathds{1}_{\{x_0,\ldots,x_n\in\set_\infty\}}(\omega)\right] 
&\leq
&\mathbb E\left[|\langle v,\varphi_\omega(x_n) - \varphi_\omega(x_1)\rangle|\cdot \mathds{1}_{\{x_0,\ldots,x_n\in\set_\infty\}}(\omega)\right] \\
&=
&\mathbb E\left[|\langle v,\varphi_{\tau_{x_1}\omega}(x_n-x_1)\rangle|\cdot \mathds{1}_{\{x_0,\ldots,x_n\in\set_\infty\}}(\omega)\right] \\
&\leq 
&\mathbb E\left[|\langle v,\varphi_{\tau_{x_1}\omega}(x_n-x_1)\rangle|\cdot \mathds{1}_{\{x_1,\ldots,x_n\in\set_\infty\}}(\omega)\right] \\
&=
&\mathbb E\left[|\langle v,\varphi_{\tau_{x_1}\omega}(x_n-x_1)\rangle|\cdot \mathds{1}_{\{0,\ldots,(x_n-x_1)\in\set_\infty\}}(\tau_{x_1}\omega)\right] \\
&=
&\mathbb E\left[|\langle v,\varphi_{\omega}(x_n-x_1)\rangle|\cdot \mathds{1}_{\{0,\ldots,(x_n-x_1)\in\set_\infty\}}(\omega)\right] \\
&\leq
&\ldots\\
&\leq
&\mathbb E\left[|\langle v,\varphi_{\omega}(x_n-x_{n-1})\rangle|\cdot \mathds{1}_{\{0,(x_n-x_{n-1})\in\set_\infty\}}(\omega)\right] \\
&=
&0,
\end{eqnarray*}
where in the first step we used the triangle inequality and \eqref{eq:scalprod=0:1}, 
in the second step we used property 3 of the corrector, in the fifth we used the shift invariance of $\mathbb P$, 
and in the last step we again used \eqref{eq:scalprod=0:1}. Thus, for any nearest neighbor path $\pi$ from $0$ to $y\in\Z^d$, 
\[
\mathbb E\left[|\langle v,\varphi_\omega(y)\rangle|\cdot \mathds{1}_{\{\pi\subset\set_\infty\}}(\omega)\right] = 0.
\]
Fix $y\in\Z^d$. By summing over all simple nearest neighbor paths $\pi$ from $0$ to $y\in\Z^d$, we obtain that 
\begin{eqnarray*}
0 &= 
&\sum_\pi\mathbb E\left[|\langle v,\varphi_\omega(y)\rangle|\cdot \mathds{1}_{\{\pi\subset\set_\infty\}}(\omega)\right] \\
&=
&\mathbb E\left[|\langle v,\varphi_\omega(y)\rangle|\cdot \sum_\pi\mathds{1}_{\{\pi\subset\set_\infty\}}(\omega)\right] \\
&\geq
&\mathbb E\left[|\langle v,\varphi_\omega(y)\rangle|\cdot \mathds{1}_{\{0,y\in\set_\infty\}}(\omega)\right],
\end{eqnarray*}
which implies \eqref{eq:scalprod=0:2}.

For $u\in\mathbb R^d$, let $[u]$ be the closest to $u$ point from $\set_\infty$ (with ties broken arbitrarily). 
By \a{}, \aaa{}, and the Borel-Cantelli lemma, for all large enough $n$, $|[nv] - nv| \leq \sqrt{n}$. 
Since $\langle v, \varphi_\omega([nv])\rangle = 0$ for $\mathbb P_0$-almost every $\omega$, 
which is equivalent to $\langle v, [nv]\rangle = - \langle v, \chi([nv],\omega)\rangle$, 
if we divide by $n$, send $n$ to infinity, and use \eqref{eq:chimaxsublinearity}, then we arrive at $\langle v,v\rangle = 0$. 

Thus, we proved that for all $v\neq 0$, $\langle v, \Sigma^2 v\rangle>0$. 
The proof of Theorem~\ref{thm:ipa} is complete, subject to the assertion of Lemma~\ref{l:chimaxsublinearity}.
\end{proof}

We will now prove Lemma~\ref{l:chimaxsublinearity}. 
\begin{proof}[Proof of Lemma~\ref{l:chimaxsublinearity}]
The proof of Lemma~\ref{l:chimaxsublinearity} follows the strategy indicated in \cite[Theorem~2.4]{BiskupPrescott}, 
where sufficient conditions for \eqref{eq:chimaxsublinearity} are stated. 
It follows from \cite{BiskupPrescott} that if the corrector satisfies the assumptions of \cite[Theorem~2.4]{BiskupPrescott} and $\mathbb P$ satisfies \cite[Proposition~2.3]{BiskupPrescott}, 
then \eqref{eq:chimaxsublinearity} holds. 
Note that in our case, \cite[Proposition~2.3]{BiskupPrescott} always holds, since (in the notation of \cite{BiskupPrescott}) $\mathcal C_\infty\setminus\mathcal C_{\infty,\alpha} = \emptyset$ 
for any $\alpha$ satisfying \cite[(2.8) and (2.9)]{BiskupPrescott}. 
Thus, it suffices to check that $\chi$ satisfies conditions of \cite[Theorem~2.4]{BiskupPrescott} (with $\mathcal C_{\infty,\alpha}$ replaced by $\set_\infty$). 
We will verify these conditions now.

The same proof as the one of \cite[Theorem~4.1(4)]{BiskupPrescott} gives that if $\mathbb P$ satisfies \a{}, \aa{}, and \aaaa{}, then 
for some $\theta>0$ and $\mathbb P_0$-almost every $\omega$,
\[
\lim_{k\to\infty}\max_{x\in\set_\infty\cap\ballZ(0,k)}\frac{|\chi(x,\omega)|}{k^\theta} = 0 .\
\]
This is condition (2.16) of \cite[Theorem~2.4]{BiskupPrescott}.

For $\omega\in\Omega_0$ and $e\in\Z^d$ with $|e|_1 = 1$, let $n_e = n_e(\omega) = \min\{k>0~:~k\cdot e\in\set_\infty\}$. 
Note that if $\mathbb P$ satisfies \a{} and \aa{}, then the set $\{k>0~:~k\cdot e\in\set_\infty\}$ has positive density in $\mathbb N$, 
and so $n_e(\omega)<\infty$ almost surely. 
\begin{lemma}\label{l:sublinearity}
Let $d\geq 2$, and $\mathbb P$ satisfies \a{} and \aa{}. 
If, in addition, for every $e\in\Z^d$ with $|e|_1 = 1$, $\mathbb E_0[|\chi(n_e,\cdot)|]<\infty$ and $\mathbb E_0[\chi(n_e,\cdot)] = 0$, then 
for all $\epsilon>0$ and $\mathbb P_0$-almost every $\omega\in\Omega_0$, 
\begin{equation}\label{eq:chisublinearity}
\limsup_{k\to\infty}\frac{1}{(2k+1)^d}\sum_{x\in\set_\infty\cap\ballZ(0,k)}\mathds{1}_{\{|\chi(x,\omega)|\geq\epsilon k\}} = 0 .\
\end{equation}
\end{lemma}
\begin{proof}[Proof of Lemma~\ref{l:sublinearity}]
The proof is a word-for-word repetition of the proof of \cite[Theorem~5.4]{BergerBiskup}. 
(It is stated in \cite{BergerBiskup} only for $d\geq 3$, but the proof goes without changes for $d=2$ too, see comments at the beginning of \cite[Section~5]{BergerBiskup}.)
Indeed, the main ingredient in the proof of \cite[Theorem~5.4]{BergerBiskup} is \cite[Theorem~4.1]{BergerBiskup} (sublinearity of the corrector along coordinate axes), 
which holds for any $\mathbb P$ satisfying \a{} and \aa{} as long as conditions of \cite[Proposition~4.2]{BergerBiskup} are satisfied. 
These are exactly the additional assumptions on $\mathbb P$ in the statement of Lemma~\ref{l:sublinearity}. 
The proof of Lemma~\ref{l:sublinearity} is complete.
\end{proof}

We next observe that if, in addition, $\mathbb P$ satisfies assumptions \aaa{} and \aaaa{}, then conditions on the moments of $\chi(n_e,\cdot)$ in Lemma~\ref{l:sublinearity} are fulfilled.
\begin{lemma}\label{l:chimoments}
Let $d\geq 2$, and $\mathbb P$ satisfies \a{} -- \aaaa{}.
Then for every $e\in\Z^d$ with $|e|_1 = 1$, $\mathbb E_0[|\chi(n_e,\cdot)|]<\infty$ and $\mathbb E_0[\chi(n_e,\cdot)] = 0$.
\end{lemma}
\begin{proof}[Proof of Lemma~\ref{l:chimoments}]
The proof is essentially the same as the proof of \cite[Proposition~4.2]{BergerBiskup}. 
The latter relies on the fact that for every $e\in\Z^d$ with $|e|_1 = 1$, 
all the $\mathbb P_0$-moments of $\rho_\set(0,n_e)$ are finite, see \cite[Lemma~4.4]{BergerBiskup}. 
In our case, this follows from assumptions \aaa{} and \aaaa{} exactly as in the proof of \cite[Lemma~4.4]{BergerBiskup}.
\end{proof}
Statement \eqref{eq:chisublinearity} is precisely condition (2.15) of \cite[Theorem~2.4]{BiskupPrescott}.
Thus, it remains to show that conditions (2.17) and (2.18) of \cite[Theorem~2.4]{BiskupPrescott} hold. 
Namely, let $(N_t)_{t\geq 0}$ be the Poisson process with jump-rate $1$, and $\mathcal Y_t = X_{N_t}$. 
If $\mathbb P$ satisfies \aaaaa{}, then for $\mathbb P_0$-almost every $\omega$, 
\begin{equation}\label{eq:bp17}
\sup_{k\geq 1}\max_{x\in\set_\infty\cap\ballZ(0,k)}\sup_{t\geq k}\frac{\mathbf E_{\omega,x}[|\mathcal Y_t - x|]}{\sqrt{t}} < \infty
\end{equation}
and
\begin{equation}\label{eq:bp18}
\sup_{k\geq 1}\max_{x\in\set_\infty\cap\ballZ(0,k)}\sup_{t\geq k}t^{\frac d2}\mathbf P_{\omega,x}[\mathcal Y_t = x] < \infty .\
\end{equation}
As in the proof of \cite[Lemma~5.6]{BiskupPrescott} (see \cite[(6.33) and (6.34)]{BiskupPrescott}), 
\eqref{eq:bp18} follows once we show that $\mathbb P_0$-almost surely, 
\begin{equation}\label{eq:ciso}
\inf_{k\geq 1}\inf\left\{
\frac{|\partial_\set A|}{|A|^{\frac{d-1}{d}}}~:~A\subset \set_\infty\cap\ballS(0,2k), |A|\geq k^{1/3}
\right\}>0
\end{equation}
and \eqref{eq:bp17} follows once we show that $\mathbb P_0$-almost surely,
\begin{equation}\label{eq:cvol}
\sup_{k\geq 1}\max_{x\in\set_\infty\cap\ballZ(0,k)}\sup_{t\geq k}\sup_{0<s\leq t^{-1/2}}\left[
s^d\sum_{y\in\set_\infty}e^{-s\rho_\set(x,y)}
\right]<\infty .\
\end{equation}
Inequality \eqref{eq:ciso} is satisfied by assumption \aaaaa{}, 
and inequality \eqref{eq:cvol} follows from the fact that $\rho_\set(x,y) \geq |x-y|_1$ for all $x,y\in\set$.

We have checked that $\chi$ satisfies all the conditions of \cite[Theorem~2.4]{BiskupPrescott}. 
Therefore, \eqref{eq:chimaxsublinearity} follows, and the proof of Lemma~\ref{l:chimaxsublinearity} is complete.
\end{proof}

\paragraph{Acknowledgements.}
We thank Noam Berger, Takashi Kumagai, and Bal\'azs R\'ath for valuable discussions, 
Alain-Sol Sznitman for constructive comments on the draft, and Yoshihiro Abe for careful reading of the draft and useful comments.
The research of RR is supported by the ETH fellowship.

\end{document}